\documentclass[10pt,twocolumn,a4paper]{wlscirep}

\usepackage{pgfplots} \pgfplotsset{compat=1.12}
\usepackage[algo2e,boxed]{algorithm2e}	
	\SetAlFnt{\footnotesize} 
	\SetAlCapFnt{\footnotesize}
	\SetAlCapNameFnt{\footnotesize}
\usepackage{graphicx, amsmath, amssymb, amsthm, ragged2e, multirow, subcaption, rotating, float}

\usepackage{tikz-network}

\newtheorem{theorem}{Theorem}
\newtheorem{proposition}{Proposition}
\newtheorem{lemma}{Lemma}
\newtheorem{corollary}{Corollary}

\title{Exploiting symmetry in network analysis}
\author[1*]{Rub\'en J.~S\'anchez Garc\'ia}
\affil[1]{Mathematical Sciences, University of Southampton, University Road, Southampton SO17 1BJ, United Kingdom}
\affil[*]{R.Sanchez-Garcia@soton.ac.uk}


\begin{abstract}
Virtually all network analyses involve structural measures between pairs of vertices, or of the vertices themselves, and the large amount of symmetry present in real-world complex networks is inherited by such measures. This has practical consequences which have not yet been explored in full generality, nor systematically exploited by network practitioners. Here we study the effect of network symmetry on arbitrary network measures, and show how this can be exploited in practice in a number of ways, from redundancy compression, to computational reduction. We also uncover the spectral signatures of symmetry for an arbitrary network measure such as the graph Laplacian. Computing network symmetries is very efficient in practice, and we test real-world examples up to several million nodes. 
Since network models are ubiquitous in the Applied Sciences, and typically contain a large degree of structural redundancy, our results are not only significant, but widely applicable.
\end{abstract}

\begin{document}

\flushbottom
\maketitle

\thispagestyle{empty}
Network models of real-world complex systems have been extremely successful at revealing structural and dynamical properties of these systems \cite{newman2010networks}. The success of this approach is due to its simplicity, versatility, and surprising universality, with common properties and principles shared by many disparate systems \cite{watts1998collective, barabasi1999emergence,milo2002network}. 

One property of interest is the presence of structural redundancies, which manifest themselves as symmetries in a network model. Symmetries relate to system robustness \cite{tononi1999measures,albert2002statistical}, as they identify structurally equivalent nodes, and can arise from replicative growth processes such as duplication \cite{chung2003duplication}, evolution from basic principles \cite{xiao2008emergence}, or functional optimisation \cite{nishikawa2016network}, and can be arbitrarily generated in model graphs\cite{Klickstein2018generating}. It has been shown that real-world networks possess a large number of symmetries \cite{macarthur2008symmetry, xiao2008network, xiao2008emergence, macarthur2009spectral, pecora2014cluster}, and that this has important consequences for network structural \cite{macarthur2008symmetry}, spectral \cite{macarthur2009spectral} and dynamical\cite{golubitsky2003symmetry, stewart2003symmetry, golubitsky2005patterns, aguiar2018synchronization, sorrentino2017symmetries} properties, for instance cluster synchronisation \cite{nicosia2013remote, pecora2014cluster, sorrentino2016complete, sorrentino2016approximate, schaub2016graph, pecora2017discovering, siddique2018symmetry}. 
 
Crucially, network symmetries are inherited by any measure or metric on the network, that is, any structural measurement between pairs of vertices (such as distances), vertex-valued measurements (such as centrality) or even matrices derived from the network (such as the graph Laplacian). However, the effects of symmetry on arbitrary network measures is not yet fully understood nor exploited in network analysis, even though the network symmetry of the large but sparse graphs typically found in applications can be effectively computed and manipulated.

In this article, we show how a network representation of an arbitrary pairwise measure inherits the same symmetries of the original network, and uncovers the structural and spectral signatures of symmetry on this network representation. Namely, for an arbitrary network measure, we identify subgraphs where the symmetry is generated (symmetric motifs) and their structure, use the network quotient to quantify the redundancy due to symmetry, develop general compression algorithms that eliminate this redundancy, and study the reduction in computational time obtained by exploiting the presence of symmetries. The eigenvalues and eigenvectors of a network measure also reflect the presence of symmetry: we show how symmetry explains most of the discrete spectrum of an arbitrary network measure, predict the most significant eigenvalues due to symmetry, and use this to develop a fast symmetry-based eigendecomposition algorithm. We achieve remarkable empirical results in our real-world test networks: compression factors up to 26\% of the original size, over 90\% of the discrete spectrum explained by symmetry, and full eigendecomposition computations in up to 13\% of the original time, demonstrating the practical use of symmetry in network analysis. We also discuss the implications of network symmetry in vertex measures. We illustrate our approach in several network measures, providing novel results of independent interest for the shortest path distance, communicability, the graph Laplacian, closeness centrality and eigenvector centrality. To facilitate dissemination, we provide full implementations of all the algorithms described in this article\cite{BitBucket}. Our results supersede \cite{macarthur2008symmetry, macarthur2009spectral} and help to understand other network symmetry results thereafter \cite{xiao2008network, xiao2009efficiently, wang2012symmetry, karalus2015symmetry, nyberg2015mesoscopic, do2012engineering}. We focus on structural and spectral properties, and symmetries commonly found in real-world networks: For a more general study of arbitrary symmetry in (networks of) dynamical systems, see\cite{golubitsky2003symmetry, stewart2003symmetry, golubitsky2005patterns, aguiar2018synchronization}. To keep our account as self-contained as possible, we include material well known in the algebraic graph theory literature e.g.~\cite{biggs1993algebraic, godsil2013algebraic, cvetkovic2010introduction, brouwer2011spectra}, without any originality claim.

\section*{Results}
\subsection*{Symmetry in complex networks}
The notion of network symmetry is captured by the mathematical concept of \emph{graph automorphism} \cite{biggs1993algebraic}. This is a permutation of the vertices (nodes) preserving adjacency, and can be expressed in matrix form using the adjacency matrix of the network. If a network  (mathematically, a finite simple graph) $\mathcal G$ has $n$ vertices, labelled $1$ to $n$, its \emph{adjacency matrix} $A=(a_{ij})$ is an $n \times n$ matrix with $(i,j)$-entry $a_{ij}=1$ if there is an edge between nodes $i$ and $j$, and zero otherwise. A graph automorphism $\sigma$ is then a permutation, or relabelling, of the vertices $v \mapsto \sigma(v)$ such that $(\sigma(i),\sigma(j))$ is an edge only if $(i,j)$ is an edge, or, equivalently, $a_{ij} = a_{\sigma(i)\sigma(j)}$ for all $i$, $j$.
In matrix terms, this can be written as 
\begin{equation}\label{eqn:autP}
	AP = PA\,, 
\end{equation}
where $P$ is the permutation matrix corresponding to $\sigma$, that is, the matrix with $(i,j)$-entry 1 if $\sigma(i)=j$, and 0 otherwise. The automorphisms of a graph form a mathematical structure called a group, the \emph{automorphism group} of $\mathcal G$.  In principle, any (finite) group $G$ is the automorphism group of some graph $\mathcal G$ \cite{biggs1993algebraic}, but, in practice, real-world networks exhibit very specific types of symmetries generated at some small subgraphs called \emph{symmetric motifs} \cite{macarthur2008symmetry}. Namely, we can partition the vertex set into the \emph{asymmetric core} of fixed points $V_0$ (an automorphism $\sigma$ \emph{moves} a vertex $i \in V$ if $\sigma(i) \neq i$, and \emph{fixes} it otherwise), and the vertex sets $M_i$ of the symmetric motifs, 
\begin{equation}\label{eqn:VertexDecomp}
	V = V_0 \cup M_1 \cup \ldots \cup M_m,
\end{equation} 
as shown in Fig.~\ref{fig:toy}a for a toy example. Equation \eqref{eqn:VertexDecomp} is called the \emph{geometric decomposition} of the network\cite{macarthur2008symmetry}. 

\begin{figure}[h!]
\centering
\includegraphics[width=1\linewidth]{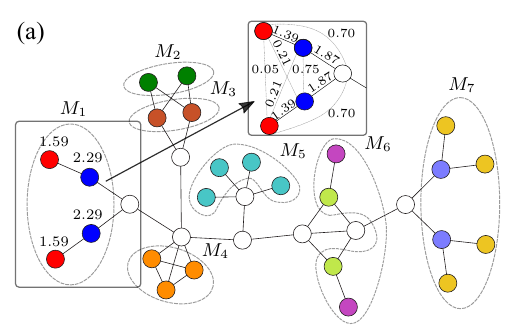}\\[3pt]
\includegraphics[width=0.94\linewidth]{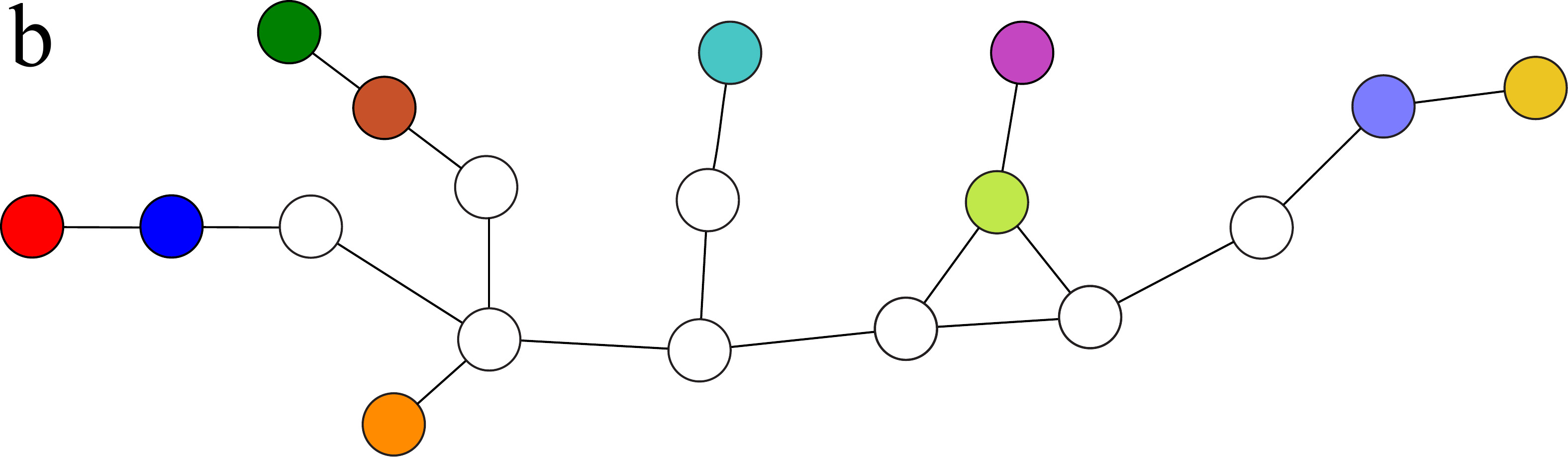}
\caption{\textbf{Toy example of a symmetric network.} \small (a) Decomposition into asymmetric core (white nodes) and 7 symmetric motifs $M_1$ to $M_7$. Each motif can be decomposed into one or more orbits (shown by colour) of structurally indistinguishable nodes. 
Motifs $M_1$ to $M_6$ are \emph{basic}: one or more orbits of the same size, and every permutation of the vertices in an orbit can be extended to a network automorphism. Motif $M_7$, a tree, is not basic: in the orbit of yellow vertices, labelled 1 to 4 (top to bottom), the permutation sending 1 to 3 must move 2 to 4 (as 1 and 2 share a common neighbour).
A vertex measure (here subgraph centrality\cite{estrada2005subgraph} shown on $M_1$) is constant on orbits. Inset: A pairwise measure (here communicability\cite{estrada2008communicability} shown near $M_1$) inherits all the network symmetries. (b) Quotient network (no loops, edge directions, or weights shown) consisting on one vertex per fixed point (white node) and per orbit (coloured node).}
\label{fig:toy}
\end{figure}

Real-world networks typically exhibit a core of fixed points (\emph{asymmetric core}), and a large number of relatively small symmetric motifs, where all the network symmetry is generated, and hence the size of the automorphism group is often extremely large, in stark contrast to random graphs, typically asymmetric \cite{erdos1963asymmetric}. However, each symmetry is the product (composition) of automorphisms permuting a very small number of vertices within a symmetric motif. For example, the toy graph in Fig.~\ref{fig:toy}a has $2^7\times3!\times4!=18,432$ symmetries (size of the automorphism group) but they generated by (all combinations of) just 10 permutations, each permuting a few vertices within a symmetric motif (one permutation per motif except two for $M_4$, $M_5$ and $M_7$). 

Each symmetric motif can be further subdivided into orbits of structurally indistinguishable nodes (shown by colour in Fig.~\ref{fig:toy}a), which play the same structural role in the network and, therefore, contribute to network redundancy and thus to the robustness of the underlying system. Our notion of structurally indistinguishable nodes (nodes in the same orbit of the automorphism group) extends the notion of structurally equivalent nodes found in the social sciences\cite{borgatti2015structural}, that is, nodes with the same set of neighbours. It is not equivalent: nodes in the same orbit may not have the same neighbours (e.g.~$M_1$, $M_6$ or $M_7$ in Fig.~\ref{fig:toy}a).

Network symmetries of (possibly very large) real-world networks can be effectively computed, stored and manipulated (see Methods). For instance, we computed generators of the automorphism group, and the subsequent geometric decomposition, for real-world networks up to several million nodes and edges in a few seconds (see $t_1$ and $t_2$ in Table \ref{Table1}). 

Most symmetric motifs in real-world networks (typically over 90\%, see the $bsm$ column in Table \ref{Table1}) are of a very specific type, called \emph{basic}\cite{macarthur2008symmetry}: they are made of one or more orbits of the same size, and every permutation of the vertices in each orbit is realisable, that is, can be extended to a network automorphism (see Fig.~\ref{fig:toy}). Basic symmetric motifs (BSMs) have a very constrained structure\cite{macarthur2009spectral}, which we will generalise to arbitrary network measures and exploit throughout this article. Non-basic symmetric motifs (typically branched trees, as $M_7$ in Fig.~\ref{fig:toy}) are called \emph{complex}; they are rare and can either be studied on a case-by-case basis, or removed from the symmetry computation altogether (by ignoring the symmetries generated by them). 

The definition of network automorphism Eq.~\eqref{eqn:autP} carries to an arbitrary $n \times n$ real matrix $A=(a_{ij})$. Any such matrix can be seen as the adjacency matrix of a network with $n$ vertices labelled 1 to $n$, and an edge (link) from node $i$ to node $j$ with weight $a_{ij}$ if $a_{ij} \neq 0$, and no such edge if $a_{ij}=0$. This means that an automorphism does not only preserve edges, but also their weights and directions. This may not be a realistic assumption for real-world weighted networks, where the weights often come from observational or experimental data, but it applies to the matrix representing a network structural measure, as we illustrate in Fig.~\ref{fig:new} and explain next. 

\begin{figure}[h!]
\centering
\begin{subfigure}{0.2\textwidth}
\hspace*{-22pt}\raisebox{-0.6cm}{\begin{tikzpicture}
	\node at (-0.52,1) {(a)};
	\Vertex[x=0,y=1,size=0.3,color=red,label=1,fontscale=1.1,style=thin]{1}
	\Vertex[x=0,y=0,size=0.3,color=red,label=2,fontscale=1.1,style=thin]{2}
	\Vertex[x=1,y=1,size=0.3,color=blue!70!white,label=3,fontscale=1.1,style=thin]{3}
	\Vertex[x=1,y=0,size=0.3,color=blue!70!white,label=4,fontscale=1.1,style=thin]{4}
	\Vertex[x=2,y=0.5,size=0.3,opacity=0,label=5,fontscale=1.1,style=thin]{5}
	\Edge[lw=0.8pt](1)(3)
	\Edge[lw=0.8pt](2)(4)
	\Edge[lw=0.8pt](3)(5)
	\Edge[lw=0.8pt](4)(5)
\end{tikzpicture}}
\end{subfigure}
\begin{subfigure}[t]{0.2\textwidth}
\scalebox{0.8}{\renewcommand*{\arraystretch}{0.9}
$A=
\begin{pmatrix}
	0&0&1&0&0\\0&0&0&1&0\\1&0&0&0&1\\0&1&0&0&1\\0&0&1&1&0
\end{pmatrix}$}
\end{subfigure}
\begin{subfigure}{0.18\textwidth}
 \hspace*{-10pt}\begin{tikzpicture}
	 \node at (-0.55,1.2) {(b)};
	\Vertex[x=0,y=1,size=0.3,color=red,label=1,fontscale=1.1,style=thin]{1}
	\Vertex[x=0,y=0,size=0.3,color=red,label=2,fontscale=1.1,style=thin]{2}
	\Vertex[x=1,y=1,size=0.3,color=blue!70!white,label=3,fontscale=1.1,style=thin]{3}
	\Vertex[x=1,y=0,size=0.3,color=blue!70!white,label=4,fontscale=1.1,style=thin]{4}
	\Vertex[x=2,y=0.5,size=0.3,opacity=0,label=5,fontscale=1.1,style=thin]{5}
	\Edge[lw=0.8pt,loopposition=90,label=1.59,fontscale=0.8,position=below](1)(1)
	\Edge[lw=0.8pt,loopposition=-90,label=1.59,fontscale=0.8,position=above](2)(2)
	\Edge[lw=0.8pt,loopposition=90,label=2.29,fontscale=0.8,position=below](3)(3)
	\Edge[lw=0.8pt,loopposition=-90,label=2.29,fontscale=0.8,position=above](4)(4)
	\Edge[lw=0.8pt,loopposition=0,label=3.42,fontscale=0.8,position=left](5)(5)
	\Edge[lw=0.8pt,label=0.05,fontscale=0.8,position=left](1)(2)
	\Edge[lw=0.8pt,label=1.39,fontscale=0.8,position=above](1)(3)
	\Edge[lw=0.8pt](1)(4)
	\Edge[lw=0.8pt,label=0.70,fontscale=0.8,position={above=0pt},distance=0.8,bend=-5,NotInBG](1)(5)
	\Edge[lw=0.8pt](2)(3)
	\Edge[lw=0.8pt,label=1.39,fontscale=0.8,position=below](2)(4)
	\Edge[lw=0.8pt,label=0.70,fontscale=0.8,position={below=-0.5pt},distance=0.8,bend=5,NotInBG](2)(5)
	\Edge[lw=0.8pt,label=0.75,fontscale=0.8,position=right](3)(4)
	\Edge[lw=0.8pt,label=1.87,fontscale=0.8,position=above,bend=35](3)(5)
	\Edge[lw=0.8pt,label=1.87,fontscale=0.8,position=below,bend=-35](4)(5)
	\draw[very thin,->,draw=black!80!white] (-0.3,0.8)  to[out=-45,in=180] (0.25,0.7);
	\node [left] at (-0.2,0.9) [style={scale=0.55}] {$0.21$};
	\draw[very thin,->,draw=black!80!white] (-0.3,0.2)  to[out=45,in=180] (0.25,0.3);
	\node [left] at (-0.2,0.1) [style={scale=0.55}] {$0.21$};
\end{tikzpicture}
\end{subfigure}
\hspace{2pt}
\begin{subfigure}{0.26\textwidth}
\scalebox{0.8}{\renewcommand*{\arraystretch}{1}
$e^A=
\begin{pmatrix}
	1.59&0.05&1.39&0.21&0.70\\0.05&1.59&0.21&1.39&0.70\\1.39&0.21&2.29&0.75&1.87\\0.21&1.39&0.75&2.29&1.87\\0.70&0.70&1.87&1.87&3.42
\end{pmatrix}$}
\end{subfigure}
\begin{subfigure}{0.2\textwidth}
\hspace*{-22pt}\begin{tikzpicture}
	\node at (-0.54,1) {(c)};
	\Vertex[x=0,y=1,size=0.3,color=red,label=1,fontscale=1.1,style=thin]{1}
	\Vertex[x=0,y=0,size=0.3,color=red,label=2,fontscale=1.1,style=thin]{2}
	\Vertex[x=1,y=1,size=0.3,color=blue!70!white,label=3,fontscale=1.1,style=thin]{3}
	\Vertex[x=1,y=0,size=0.3,color=blue!70!white,label=4,fontscale=1.1,style=thin]{4}
	\Vertex[x=2,y=0.5,size=0.3,opacity=0,label=5,fontscale=1.1,style=thin]{5}
	\Edge[lw=0.8pt,loopposition=90,label=1,Math,fontscale=1,position=below](1)(1)
	\Edge[lw=0.8pt,loopposition=-90,label=1,Math,fontscale=1,position=above](2)(2)
	\Edge[lw=0.8pt,loopposition=90,label=2,Math,fontscale=1,position=below](3)(3)
	\Edge[lw=0.8pt,loopposition=-90,label=2,Math,fontscale=1,position=above](4)(4)
	\Edge[lw=0.8pt,loopposition=0,label=2,Math,fontscale=1,position=left](5)(5)
	\Edge[lw=0.8pt,label=-1,Math,fontscale=1,position=above](1)(3)
	\Edge[lw=0.8pt,label=-1,Math,fontscale=1,position=below](2)(4)
	\Edge[lw=0.8pt,label=-1,Math,fontscale=1,position=above](3)(5)
	\Edge[lw=0.8pt,label=-1,Math,fontscale=1,position=below](4)(5)
\end{tikzpicture}
\end{subfigure}
\begin{subfigure}{0.2\textwidth}
\hspace*{-10pt}\scalebox{0.8}{\renewcommand*{\arraystretch}{0.9}
$L=
\begin{pmatrix}
	1&0&-1&0&0\\0&1&0&-1&0\\-1&0&2&0&-1\\0&-1&0&2&-1\\0&0&-1&-1&2
\end{pmatrix}$}
\end{subfigure}
\caption{\textbf{Structural network measures.} \small (a) Neigh\-bour\-hood of motif $M_1$ (Fig.~\ref{fig:toy}) and its adjacency matrix. Network and matrix representation of two structural (full, respectively sparse) network measures, (b) the exponential matrix (a measure of vertex communicability\cite{estrada2008communicability}), and (c) the graph Laplacian. In each case, the vertices are numbered 1 to 5 and coloured by orbit, and a non-zero $(i,j)$-entry in the matrix corresponds to the edge weight between vertices $i$ and $j$. Note that the network automorphism swapping vertices 1 and 2, and 3 and 4, preserves the network and matrix structure in all three cases. In general, any symmetry of a network is also a symmetry of the network representation of any structural network measure. 
}
\label{fig:new}
\end{figure}

\begin{table*}
\small \centering
\begin{tabular}{l rr r cc rrr rrr rrr}
Name 		& $n_\mathcal{G}$ 		& $m_\mathcal{G}$ 			& $gen$ 		& $t_1$ 	& $t_2$	& $sm$ & $bsm$	& $mv$ & $\widetilde{n}_\mathcal{Q}$ & $\widetilde{m}_\mathcal{Q}$	& $ext_s$ & $int_f$ & $c_\text{full}$ & $sp$\\
\hline
HumanDisease	& 1,419 		& 2,738		& 713		& 0.00 	& 0.16	& 272 	& 96.0
	& 71.0 	& 48.3		& 50.4	& 83.3 	& 10$^{-3}$ & 27.2	& 11.3\\
Yeast		& 1,647		& 2,736		& 380		& 0.00 	& 0.01	& 149 	& 99.3
	& 33.3 	& 76.3	&  83.5 	& 98.8 	& 10$^{-3}$	& 58.4	& 44.4\\
OpenFlights	& 3,397		& 19,230		& 732		& 0.00 	& 0.11	& 321 	& 93.5
	& 32.4 	& 77.3	& 94.4	& 99.3 	& 10$^{-3}$	& 63.5	& 46.2\\
USPowerGrid	& 4,941		& 6,594		& 414		& 0.00 	& 0.09	& 302 	& 97.4		
	& 16.7 	& 90.2	& 91.3	& 97.6 	& 10$^{-4}$	& 83.9	& 73.3\\
HumanPPI		& 9,270		& 36,918 		& 972		& 0.00 	& 0.12	& 437 	& 100		
	& 15.3 	& 89.5	& 97.0	& 99.9 	& 10$^{-4}$	& 80.1	& 71.6\\
Astro-Ph		& 17,903		& 196,972	& 3,232		& 0.01 	& 0.21	& 1,682 	& 99.4		
	& 27.5 	& 81.9	& 80.4	& 95.5 	& 10$^{-4}$	& 67.4	& 54.9\\
InternetAS	& 34,761		& 107,720	& 15,587		& 0.03 	& 0.29	& 3,189 	& 99.9		
	& 54.3 	& 55.0	& 78.2	& 99.9 	& 10$^{-5}$	& 30.3	& 16.7\\
WordNet		& 145,145	& 656,230	& 52,152		& 0.18 	& 0.62	& 28,456 	& 92.0	
	& 60.0 	& 60.1	& 58.0	& 89.9 	& 10$^{-5}$	& 49.3	& 21.6\\
Amazon		& 334,863 	& 925,872	& 32,098		& 0.20 	& 0.39	& 23,302 	& 99.8	
	& 16.8 	& 90.3	& 89.0	& 99.0 	& 10$^{-6}$	& 81.6	& 73.6\\
Actors		& 374,511 	& 15,014,839	& 182,803	& 0.95 	& 1.38	& 36,703 	& 99.9		
	& 58.6 	& 51.2	& 66.4	& 90.4 	& 10$^{-5}$	& 26.2	& 13.4\\
InternetAS-skitter	& 1,694,616	& 11,094,209	& 319,738	& 1.71 	& 4.17	& 84,675 	& 99.1 
	& 19.7 	& 85.4	& 92.8	& 99.9 	& 10$^{-6}$	& 73.5	& 62.3\\
CaliforniaRoads	& 1,957,027	& 2,760,388	& 36,430		& 0.47 	& 0.16	& 35,210 	& 98.8
	& 4.0 	& 97.9	& 98.4	& 99.7 	& 10$^{-7}$	& 96.3	& 93.9\\
LiveJournal	& 5,189,808	& 48,687,945	& 410,575	& 8.02 	& 3.59	& 245,211 & 99.9
	& 12.7 	& 92.1	& 96.5	& 99.7 	& 10$^{-7}$	& 84.8	& 78.0\\	
\hline
\end{tabular}
\caption{\textbf{Symmetry in some real-world networks.} \small For each test network, we show the number of vertices ($n_\mathcal{G}$), edges ($m_\mathcal{G}$), number of generators ($gen$) of the automorphism group (sizes, $10^{153}$ to $10^{197,552}$, not shown), computing times of generators ($t_1$) and geometric decomposition ($t_2$), in seconds, number of symmetric motifs ($sm$) and proportion of basic symmetric motifs ($bsm$), proportion of vertices moved by an automorphism ($mv$), proportion of vertices ($\widetilde{n}_\mathcal{Q}=n_\mathcal{Q}/n_\mathcal{G}$) and edges ($\widetilde{m}_\mathcal{Q}=m_\mathcal{Q}/m_\mathcal{G}$) in the quotient, proportion of external edges in the sparse case ($ext_s$, in percentage), and of internal edges in the full case ($int_f$, closest power of 10), full compression ratio ($c_\text{full} = \widetilde{n}_\mathcal{Q}^2$), and spectral computational reduction ($sp = \widetilde{n}_\mathcal{Q}^3$), all for the largest connected component. The proportion of vertices in the \emph{basic} quotient ($\widetilde{n}_\mathcal{Q_\text{basic}}$, not shown) is within 1\% of $\widetilde{n}_\mathcal{G}$ except for HumanDisease ($\widetilde{n}_\mathcal{Q_\text{basic}}=52.2\%$), OpenFlights ($79.7\%$), USPowerGrid ($91.6\%$) and WordNet ($79.2\%$), and similar results hold for $\widetilde{m}_\mathcal{Q_\text{basic}}$. 
Datasets available at \cite{konect2016topology}, except HumanDisease \cite{HumanDisease}, Yeast \cite{Yeast}, and HumanPPI \cite{HumanPPI}. Computations on a desktop computer (3.2 GHz Intel Core i5 processor, 16 GB 1.6 GHz DDR3 memory). All networks are symmetric, although the amount of symmetry (as measured by $mv$ or $\widetilde{n}_\mathcal{Q}$) ranges from several networks with $50\%$ quotient reduction, to CalifornialRoads with only 4\% of vertices participating in any symmetry. However, the effect of compression and computational reduction multiplies as e.g.~$c_\text{full}=\widetilde{n}_\mathcal{Q}^2$ and $sp=\widetilde{n}_\mathcal{Q}^3$, achieving significant results for most of our test networks.}
\label{Table1}
\end{table*}

\subsection*{Structural network measures}
A \emph{(pairwise) structural network measure} is a function $F(i,j)$ on pairs of vertices which satisfies 
\begin{equation}\label{eq:StructuralFunction}
	F(\sigma(i),\sigma(j)) = F(i,j)\	\text{ for all } i, j \in V
\end{equation}
for all automorphisms $\sigma \in \text{Aut}(\mathcal G)$. Since automorphisms identify structurally indistinguishable vertices ($i$ and $\sigma(i)$) and, similarly, edges ($(i,j)$ and $(\sigma(i),\sigma(j))$), structural network measures are (edge) functions that depend on the network structure alone, and not, for example, on node or edge labels, or other meta-data. Most network measures are structural, including graph metrics (e.g.~shortest path), and matrices algebraically derived from the adjacency matrix (e.g.~Laplacian matrix). (We identify matrices $M$  with pairwise measures via $F(i,j)=[M]_{ij}$.) In particular, structural measures are independent of the ordering or labelling of the vertices. In contrast, functions depending, explicitly or implicitly, on some vertex ordering or labelling, are not structural, for example the shortest path length through a given node. 
Our results can be adapted to the presence of node or edge labels, or weights, by restricting to automorphisms preserving the additional structure. For simplicity, here we discuss the unlabelled case only. 

We can encode a structural measure $F$ as a network with adjacency matrix $[F(A)]_{ij}=F(i,j)$ (see Fig.~\ref{fig:new}b and c for two examples), and write \eqref{eq:StructuralFunction} in matrix form as  
\begin{equation}\label{eq:StructuralFunctionMatrix}
	F(A)\, P = P\, F(A), 
\end{equation}
where $P$ is the permutation matrix corresponding to $\sigma$. Comparing this to Eq.~\eqref{eqn:autP}, we see that the network representation of $F$, $F(\mathcal G)$, with adjacency matrix $F(A)$, inherits all the symmetries of $\mathcal G$. In particular, the network $F(\mathcal G)$ has the same decomposition into symmetric motifs Eq.~\eqref{eqn:VertexDecomp}, and orbits, as $\mathcal G$. The BSMs in $F(\mathcal G)$ must occur on the same vertices $M_i$, although they are now all-to-all weighted subgraphs in general (Fig.~\ref{fig:new}b). Nevertheless, they have a very constrained structure: the intra and inter orbit connectivity depends on two parameters only. Namely, each orbit in a BSM is uniquely determined by $\beta=F(v_i,v_i)$ (the connectivity of a vertex with itself) and $\alpha=F(v_i,v_j), i \neq j $ (the connectivity of a vertex with every other vertex in the orbit), for all $v_i, v_j$ in the orbit. Similarly, the connectivity between two orbits $\Delta_1$ and $\Delta_2$ in the same BSM also depends on two parameters: after a suitable reordering $\Delta_1=\{v_1,\ldots,v_n\}$ and $\Delta_2=\{w_1,\ldots,w_n\}$, we have $\delta=F(v_i,w_i)$ and $\gamma=F(v_i,w_j)$ for all $1\le i,j \le n$. (For a proof, see Theorem \ref{thm:BSM} in Methods.) This can be observed in Fig.~\ref{fig:new}c and is represented schematically in Fig.~\ref{fig:uniform}a and b. In particular, each BSM takes a very constrained form in the quotient, as shown schematically in Fig.~\ref{fig:uniform}c and d.

The results in this article apply to arbitrary structural measures, although the two most common cases in practice are the following. We call $F$ \emph{full} if $F(i,j) \neq 0$ for all $i \neq j \in V$ (e.g.~a graph metric), and \emph{sparse} if $F(i,j)=0$ if $a_{ij}=0$, for all $i \neq j \in V$ (e.g.~the graph Laplacian). The graph representation of $F(\mathcal G)$ is an all-to-all weighted graph if $F$ is full, and has a sparsity similar to $\mathcal{G}$ if $F$ is sparse (cf.~Fig.~\ref{fig:new}c).

From now on, we will assume that $\mathcal{G}$ is undirected and $F$ is symmetric, $F(i,j)=F(j,i)$, which may not be the case even if $\mathcal G$ is undirected (e.g.~the transition probability of a random walker $F(i,j)=\tfrac{a_{ij}}{\text{deg}(i)}$), and discuss directed networks and asymmetric measures in the Methods section. 

\begin{figure}
\begin{center}
\includegraphics[width=0.66\linewidth]{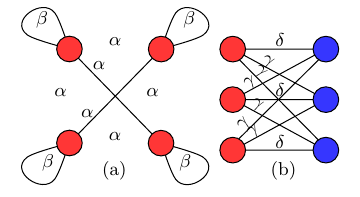}
\raisebox{4.5pt}{\begin{tikzpicture}[scale=1]
	\Vertex[x=0.5,y=1.3,size=0.4,color=red!80!white,style=thin]{1}
	\Edge[lw=0.5pt,loopposition=90](1)(1)
	\node [above] at (0.5,1.9) {\scriptsize $(n-1)\alpha+\beta$};
	\node [below] at (0.5,1.15)  {\scalebox{0.9}{(c)}};
	
	\Vertex[x=-0.2,y=0.05,size=0.4,color=red!80!white,style=thin]{2}
	\Vertex[x=1.2,y=0.05,size=0.4,color=blue!85!white,style=thin]{3}
	\Edge[lw=0.5pt](2)(3)	
	\node [above] at (0.5,0.12) {\scriptsize $(n-1)\gamma+\delta$};
	\node [below] at (0.5,0.0)  {\scalebox{0.9}{(d)}};
\end{tikzpicture}}
\end{center}
\caption{\textbf{Structure of a basic symmetry motif (BSM) for an arbitrary network measure $F$.} \small (a) Every orbit in a BSM is an \emph{$(\alpha,\beta)$-uniform graph} $K_n^{\alpha,\beta}$, the graph with $n$ vertices and adjacency matrix $A=(a_{ij})$ with $a_{ij}=\alpha=F(i,j)$ if $i\neq j$ and $a_{ii}=\beta=F(i,i)$ for some constants $\alpha$ and $\beta$. Here we show an example of one orbit with four nodes and edges labelled by their weights. (b) The connectivity between two orbits $\Delta_1$ and $\Delta_2$ in the same BSM (after a suitable relabelling $\Delta_1 = \{v_1,\ldots,v_n \}$, $\Delta_2 = \{w_1,\ldots,w_n \}$) is given by $\gamma=F(v_i,w_j)$ for $i \neq j$, and $\delta=F(v_i,w_i)$, the \emph{$(\delta,\gamma)$-uniform join} of the two orbits. Here we show an example of two orbits (shown by colour) with three vertices each and edges labelled by their weights. (c) In the quotient, the BSM orbit becomes a single vertex with a self-loop weighted by $(n-1)\alpha+\beta$, and (d) the two orbits are joined by an edge weighted by $(n-1)\gamma+\delta$. 
Here we show the quotients (c) and (d) of the previous BSMs (a) and (b) respectively. Note that, by annotating each orbit in the quotient by $n$ and $\alpha$ (or $\beta$), and each intra-motif edge by $\gamma$ (or $\delta$), we can recover each BSM from such annotated quotient.}
\label{fig:uniform}
\end{figure}

\subsection*{Quotient network}
The formal procedure to quantify and eliminate structural redundancies in a network is via its \emph{quotient network}. This is the graph with one vertex per orbit or fixed point (see Fig.~\ref{fig:toy}b) and edges representing average connectivity. Formally, if $A$ is the $n \times n$ adjacency matrix of a graph $\mathcal G$, the \emph{quotient network} with respect to a partition of the vertex set $V=V_1 \cup \ldots \cup V_m$ is the graph $\mathcal Q$ with $m \times m$ adjacency matrix the \emph{quotient matrix} $Q(A)=(b_{kl})$ defined by  
\begin{equation}\label{eqn:quotient}
	b_{kl} = \frac{1}{|V_k|} \sum_{\substack{ i \in V_k \\ j \in V_l}} a_{ij},
\end{equation}
the average connectivity from a vertex in $V_k$ to all vertices in $V_l$. 
There is an explicit matrix equation for the quotient. Consider the $n \times m$ \emph{characteristic matrix} $S$ of the partition, that is, $[S]_{ik}=1$ if $i \in V_k$, and zero otherwise, and the diagonal matrix $\Lambda=\textup{diag}(n_1,\ldots,n_m)$, where $n_k=|V_k|$. Then
\begin{equation}\label{eq:QuotientMatrixDef}
	Q(A) = \Lambda^{-1} S^T A\, S.
\end{equation}
The quotient network is a directed and weighted network in general. An alternative is to use the \emph{symmetric quotient}, with adjacency matrix $Q_\text{sym}(A)=\Lambda^{-1/2} S^T A\, S\Lambda^{-1/2}$, which is weighted but undirected. Note that $Q(A)$ and $Q_\text{sym}(A)$ are spectrally equivalent matrices: they have the same eigenvalues, with eigenvectors related by the transformation $\mathbf{v} \mapsto \Lambda^{1/2}\mathbf{w}$. 

In the context of symmetries, we will always refer to the quotient with respect to the partition of the vertex set into orbits. This quotient removes all the original symmetries from the network: if $\sigma(v_i)=v_j$, then $v_i$ and $v_j$ are in the same orbit and hence represented by the same vertex in the quotient network, which is then fixed by $\sigma$. We can, therefore, infer and quantify properties arising from redundancy alone by comparing a network to its quotient. The quotients of real-world networks are often significantly smaller (in vertex and edge size) than the original networks\cite{macarthur2008symmetry, xiao2008network} (see $\widetilde{n}_\mathcal{Q}$ and $\widetilde{m}_\mathcal{Q}$ in Table \ref{Table1}), and this reduction quantifies the structural redundancy present in an empirical network. Not every real-world network is equally symmetric, and, in our test networks, we give examples of network quotient reductions ranging from about 50\% to just 2\%. 
Computing the network quotient involves multiplication by very sparse matrices ($\Lambda$ is diagonal and $S$ has one non-zero element per row) and hence is computationally efficient (a few seconds in all our test networks).

\begin{figure}
\begin{tikzpicture}[scale=0.95]
\begin{axis}[
	xlabel={Size $(n_\mathcal{G}^2)$},
	ylabel={Compression ratio $\left(n_\mathcal{Q}^2/n_\mathcal{G}^2\right)$},
	xmode=log,
	legend pos = south east,
	legend style = {font=\small},
	xmin=10^5.2,
	xmax=10^14.4,
	ymin=0, 
	ymax=105,
	ytick = {0, 10, 20, 30, 40, 50, 60, 70, 80, 90, 100},
	yticklabel=\pgfmathprintnumber{\tick}\,\%,
]
\addplot+ [	
	only marks, 
	mark=o, 
	point meta=explicit symbolic, 
	nodes near coords, 
	every node near coord/.append style={font=\scriptsize, color=black},
] 
table [ 
	x expr=\thisrowno{0}^2, 
	meta=network, 
] {compression_predicted.data};
\addplot+ [	
	only marks, 
	mark=x, 
] 
table [ 
	x expr=\thisrowno{0}^2,  
] {compression_actual.data};
\legend{predicted, actual}
\end{axis}
\end{tikzpicture}
\caption{\textbf{Redundancy in some real-world networks.} \small 
Amount of symmetry present on a typical (full) network measure for our test networks (Table \ref{Table1}). The predicted values, $c_\text{full} = \left(n_\mathcal{G}/n_\mathcal{Q}\right)^2$ coincide (up to $0.01\%$) with the compression ratio (shown as `actual') obtained by our lossless compression algorithm for the shortest path distance, and communicability (exponential matrix), for our smallest seven networks (memory limit in our computer). After decompression, we recover the original matrix exactly for the shortest path distance, and up to a small numerical error ($1.16 \times 10^{-4}$ mean relative error) for communicability.}
\label{fig:compression}
\end{figure}

\subsection*{Redundancy in network measures} 
The amount structural redundancy on a network (measured by $\widetilde{n}_\mathcal{Q}=n_\mathcal{G}/n_\mathcal{Q}$) is amplified in the computation of a typical (full) network measure (see Eq.~\eqref{eq:cfull} below). It is therefore natural to ask how to quantify, and eliminate, the symmetry-induced redundancy. If a network has $n_\mathcal{G}$ vertices and $n_\mathcal{Q}$ orbits, there are $n_\mathcal{G}^2$ pairs of vertices but only $n_\mathcal{Q}^2$ pairs of orbits, achieving a reduction, or compression ratio, of
\begin{equation}\label{eq:cfull}
	c_\text{full}=\left(\frac{n_\mathcal{G}}{n_\mathcal{Q}}\right)^2
\end{equation}
for a full network measure, typically much smaller than the ratio $\widetilde{n}_\mathcal{Q}=n_\mathcal{G}/n_\mathcal{Q}$. On the other hand, for a sparse network measure, we only need to consider edge values, hence the reduction is the ratio between the number of edges in the graph and in its quotient
\begin{equation}
	c_\text{sparse}=\frac{m_\mathcal{G}}{m_\mathcal{Q}}.
\end{equation}
For an arbitrary network measure, its compression ratio, which measures the redundancy present (zero values excluded), will range between $c_\text{full}$ and $c_\text{sparse}$. The compression ratios $c_\text{full}$ and $c_\text{sparse}=\widetilde{m}_\mathcal{Q}$ are shown on Table \ref{Table1} for our test networks. We found a remarkable amount of redundancy (up to 70\%) due to symmetry alone (Fig.~\ref{fig:compression}). 


\subsection*{Symmetry compression}
A natural question, with practical consequences for network analysis, is whether we can easily `eliminate' the symmetry-induced redundancies. This means storing only one value of a network function for each orbit of structurally indistinguishable nodes or edges, all sharing the same such value. Although this has been explored in particular cases, such as shortest path distances\cite{xiao2009efficiently}, here we present a general treatment. 
A simple method is to use the quotient matrix
\begin{equation}
	B=S^TAS,
\end{equation}
which is easier to store than $\Lambda^{-1}S^TAS$. This matrix achieves a compression ratio between $c_\text{full}$ and $c_\text{sparse}$ (by using a sparse representation of $B$), as explained before. From this matrix, we can recover all but the internal connectivity inside a symmetric motif, which is replaced by the average connectivity. Namely, let us define 
\begin{equation}
	\overline{a}_{ij} = \frac{1}{n_i}\frac{1}{n_j}b_{kl},
\end{equation}
where $n_i$, respectively $n_j$, is the size of the orbit containing $v_i$, respectively $v_j$ (note that these orbit sizes can be obtained as the row sums of the characteristic matrix $S$). Then one can show (Methods, Theorem \ref{thm:avgcomp}) that
\begin{equation}
\overline{a}_{ij} =
\begin{cases}
	  a_{ij} & \text{if $v_i$ and $v_j$ are external,}\\
	  	\frac{1}{n_i}\frac{1}{n_j} \sum_{\substack{v_k \in \Delta_1\\v_l \in \Delta_2}} a_{kl}& \text{if $v_i$ and $v_j$ are internal,}\\
\end{cases}
\end{equation}
where we call a pair of vertices \emph{external} if they belong to two different symmetric motifs, and \emph{internal} otherwise, and $v_i\in\Delta_1$ and $v_1\in\Delta_l$ are orbits. Hence, if we are not interested in the exact internal connectivity (inside a symmetric motif), or it can be recovered easily by other means (e.g.~one motif at a time), we can use this simple method to eliminate all the symmetry-induced redundancies on an arbitrary network measure encoded as a matrix $A$. We have included simple \emph{average symmetry compression} and \emph{decompression} algorithms (Algs.~\ref{alg:avgcompression}, \ref{alg:avgdecompression}), where $A_\text{avg}$ is the matrix with entries $\overline{a}_{ij}$. The original $n_{\mathcal G}\times n_{\mathcal G}$ matrix $A$ is stored using the $n_{\mathcal Q} \times n_{\mathcal Q}$ quotient matrix $B$ plus a very sparse ($n$ non-zero elements) characteristic matrix $S$. 

\begin{algorithm2e}
\SetAlgoLined
\DontPrintSemicolon
\SetKwFunction{Sum}{sum}\SetKwFunction{Diag}{diag}
\KwIn{adjacency matrix $A$, characteristic matrix $S$}
\KwOut{quotient matrix $B$}
\BlankLine
$B \leftarrow S^TAS$\;
\caption{Average symmetry compression.}
\label{alg:avgcompression}
\end{algorithm2e}

\begin{algorithm2e}
\SetAlgoLined
\DontPrintSemicolon
\SetKwFunction{Sum}{sum}\SetKwFunction{Diag}{diag}
\KwIn{quotient matrix $B$, characteristic matrix $S$}
\KwOut{adjacency matrix $A_\text{avg}$}
\BlankLine
$\Lambda \leftarrow \Diag(\Sum(S))$\;
$R \leftarrow S\Lambda^{-1}$\;
$A_\text{avg} \leftarrow RBR^T$\;
\caption{Average symmetry decompression.}
\label{alg:avgdecompression}
\end{algorithm2e}

The vast majority of edges in the network representation of a network measure are external (at least 99.999\% for a full measure in our test networks, see $int_f$ in Table \ref{Table1}), and hence the information loss by using $A_\text{avg}$ instead of $A$ is minimal. We can nevertheless enforce \emph{lossless compression}, by storing the intra-motif connectivity separately. Indeed, we can exploit the fact that most symmetric motifs in empirical networks are basic, and hence each orbit, or pair of orbits, is uniquely determined by two parameters (Fig.~\ref{fig:uniform}). If we disregard the symmetries generated at non-basic symmetric motifs, the corresponding quotient, called \emph{basic quotient}, written $\mathcal{Q}_\text{basic}$, leaves non-basic motifs unchanged and retains most of the symmetry in a typical real-world network. By annotating this quotient, we can recover the original network representation of the network measure exactly. We have implemented lossless compression and decompression algorithms (Methods, Algs.~\ref{alg:llcompression} and \ref{alg:lldecompression}), and evaluated them in our test networks (Fig.~\ref{fig:compression}).

\begin{figure}[h!]
\begin{tikzpicture}[scale=0.95]
\begin{axis}[
	ylabel={Computational reduction},
	xmajorgrids=true,
	xmin=0.5,
	xmax=7.5,
	xtick = {1,2,3,4,5,6,7},
	xticklabels={spectral\,, commun, laplacian, distance, closeness, btwness, eigc},
	xticklabel style={
	font=\small,
        rotate=40,
        anchor=east,},
	ymin=0, 
	ymax=105,
	ytick = {0, 10, 20, 30, 40, 50, 60, 70, 80, 90, 100},
	yticklabel=\pgfmathprintnumber{\tick}\,\%,
	legend entries = {HumanDisease, Yeast, OpenFlights, USPowerGrid, HumanPPI, Astro-Ph, InternetAS, WordNet, Amazon, Actors},
	legend columns = 5,
	legend style = {font=\tiny, at={(0.05,-0.32)}, anchor=south west,}, 
	legend cell align=left,
	scatter/classes={%
	HumanDisease={mark=o,red,scale=1.5},
	Yeast={mark=triangle,orange,scale=1.5},%
	OpenFlights={mark=square,magenta,scale=1.5},%
	USPowerGrid={mark=diamond,pink,scale=1.75},%
	HumanPPI={mark=pentagon,purple,scale=1.75},%
	Astro-Ph={mark=o,blue,scale=1.5},%
	InternetAS={mark=triangle,cyan,scale=1.5},
	WordNet={mark=square,olive,scale=1.5},%
	Amazon={mark=diamond,teal,scale=1.5},%
	Actors={mark=pentagon,brown,scale=1.5}}
]
\addplot+ [
	only marks,
	scatter,	
	scatter src = explicit symbolic,
] 
table [
	x=test,
	y=ratio,
	meta=network,
] {
test	ratio		network
0.85	32.32	HumanDisease
0.85	52.37	Yeast
0.85	48.09	OpenFlights
0.85	78.79	USPowerGrid
0.85 	77.26	HumanPPI
1.15	13.11	HumanDisease
1.15	42.02	Yeast
1.15	44.93	OpenFlights
1.15	75.21	USPowerGrid
1.15	73.18	HumanPPI
2	13.04	HumanDisease
2	44.10	Yeast
2	45.97	OpenFlights
2	76.47	USPowerGrid
2	73.24	HumanPPI
3	11.41	HumanDisease
3	43.80	Yeast
3	49.45	OpenFlights
3	75.15	USPowerGrid
3	72.43	HumanPPI
4	29.47	HumanDisease
4	60.26	Yeast
4	68.44	OpenFlights
4	84.30	USPowerGrid
4	81.70	HumanPPI
4	64.28	Astro-Ph	
4	20.09	InternetAS
5	34.85	HumanDisease
5	62.24	Yeast
5	68.43	OpenFlights
5	85.02	USPowerGrid
5	80.17	HumanPPI
5	67.95	Astro-Ph	
5	32.72	InternetAS
6	34.44	HumanDisease
6	64.26	Yeast
6	71.51	OpenFlights
6	85.25	USPowerGrid
6	83.80	HumanPPI
6	67.80	Astro-Ph	
6	29.41	InternetAS
7	57.83	HumanDisease
7	63.48	Yeast
7	57.48	OpenFlights
7	85.41	USPowerGrid
7	85.76	HumanPPI
7	76.04	Astro-Ph	
7	49.84	InternetAS
7	58.51	WordNet			
7	80.49	Amazon			
7	56.84	Actors
};
\end{axis}
\end{tikzpicture}
\caption{\textbf{Quotient computational reduction.} \small Computational time reduction of several structural measures in some of our test networks (Table \ref{Table1}) obtained by performing the calculation in the quotient network versus the original network. The computations are: spectral decomposition of the adjacency matrix $A$ ({spectral}), exponential matrix $\text{exp}(A)$ ({commun}), pseudoinverse of the Laplacian matrix ({laplacian}), shortest path distance ({distance}), closeness centrality ({closeness}), betweenness centrality ({btwness}) and eigenvector centrality ({eigc}), using MATLAB R2018a built-in functions. For {spectral}, we also show (left column) the reduction including the (sequential) symmetric motif calculation. In each case, median computational reduction over at least 10 iterations shown.
}
\label{fig:reduction}
\end{figure}

\subsection*{Computational reduction}
Network symmetries can also reduce the computational time of evaluating an arbitrary network measure $F$. By Eq.~\eqref{eq:StructuralFunction}, we only need to evaluate $F$ on orbits, resulting in a computational reduction ratio of between $\widetilde{m}_\mathcal{Q}$ and $\widetilde{n}_\mathcal{Q}^2$ (Table \ref{Table1}) for sparse, respectively full, network measures. Of course, this assumes that the computation on each pair of vertices $F(i,j)$ is independent of one another, which is often not the case. Moreover, the calculation of $F(i,j)$ is still performed on the whole network $\mathcal G$. 

A more substantial computational reduction can be obtained by evaluating $F$ on the (often much smaller) quotient network instead. We call $F$ \emph{quotient recoverable} if it can be applied to the quotient network $\mathcal Q$, and $F(\mathcal G)$ can be recovered from $F(\mathcal Q)$, for all networks $\mathcal G$. Note that this may involve, beyond evaluating $F(\mathcal Q)$, an independent (hence parallelizable) computation on each symmetric motif (typically a very small graph). By evaluating $F$ in the quotient network, we can obtain very substantial computational time savings, depending on the amount of symmetry present and the computational complexity of $F$. Depending on the network measure, it may not be possible to recover $F(\mathcal G)$ exactly from $F(\mathcal Q)$, but only partially.
We call a network measure $F$ \emph{partially quotient recoverable} if it can be applied to a quotient network $\mathcal Q$ of a network $\mathcal G$, and all the external edges of $F(\mathcal G)$ can be recovered from $F(\mathcal Q)$, for all networks $\mathcal G$. Since the quotient averages the network connectivity, we can often recover the average values of $F$ within symmetric motifs. We call $F$ \emph{average quotient recoverable} if, in addition to external edges, the average intra-motif edges can be recovered from $F(\mathcal Q)$. A typical situation is when $F(\mathcal Q)$ equals the quotient of $F$, that is, in symbols,
\begin{equation}\label{eqn:avgqr}
	F(\mathcal Q)=Q(F(\mathcal G)).
\end{equation}

In the Applications section, we will show that communicability is average quotient recoverable, and shortest path distance is partially, but not average, quotient recoverable. Not every measure can be (partially) recovered from the quotient, for example the number of distinct paths between two vertices, as the internal connectivity within each symmetric motif is lost, and replaced by its average connectivity, in the quotient. Note that the word `partially' can be misleading: typically almost all edges are external (see $ext_s$ and $int_f$ in Table \ref{Table1}). 
The resulting computational time reduction obtained by evaluating $F$ in the quotient can be very substantial, as illustrated by several popular network measures in our test networks (Fig.~\ref{fig:reduction}).

\subsection*{Spectral signatures of symmetry}
The spectrum of the network's adjacency matrix relates to a multitude of structural and dynamical properties\cite{newman2010networks}. The presence of symmetries is reflected in the spectrum of the network\cite{macarthur2009spectral}, and indeed in the spectrum of any network measure. Symmetries give rise to high-multiplicity eigenvalues (shown as `peaks' in the spectral density) and, in fact, we can explain and predict most of the discrete part of the spectrum of an arbitrary network measure on a typical real-world network.


Let $A$ be the $n \times n$ adjacency matrix of a (possibly weighted) network (such as the network representation of a network measure). First, note that symmetry naturally produces high-multiplicity eigenvalues, since
\begin{equation}
	AP\mathbf{v}=PA\mathbf{v}=\lambda P\mathbf{v}.
\end{equation}
where $(\lambda,\mathbf{v})$ is an eigenpair of $A$ and $P$ the permutation matrix of a network automorphism (Eq.~\eqref{eqn:autP}). This gives another eigenpair $(\lambda,\mathbf{v})$ whenever $\mathbf{v}$ and $P\mathbf{v}$ are linearly independent (obviously not always the case). 

Let $B=Q(A)$ be the $m \times m$ quotient of $A$ (Eq.~\eqref{eq:QuotientMatrixDef}) with respect to the partition of the vertex set into orbits. This partition satisfies a regularity condition called \emph{equitability}\cite{brouwer2011spectra}, which can be written in matrix form as $AS=SB$, where $S$ is the characteristic matrix of the partition. In particular, if $(\lambda,\mathbf{v})$ is a quotient eigenpair, then $(\lambda,S\mathbf{v})$ is a parent eigenpair,
\begin{equation}
	A(S\mathbf{v})=SB\mathbf{v}=\lambda(S\mathbf{v}).
\end{equation}
In fact, one can show (Methods, Theorem \ref{thm:spectraldecomposition}) that $A$ has an eigenbasis of the form 
\begin{equation}
	\{ S\mathbf{v}_1,\ldots,S\mathbf{v}_m,\mathbf{w}_1,\ldots,\mathbf{w}_{n-m}\},
\end{equation}
where $\{ \mathbf{v}_1,\ldots, \mathbf{v}_m\}$ is any eigenbasis of $B$, and $S^T\mathbf{w}_j=0$ for all $j$. We can think of a vector $\mathbf{v}\in \mathbb{R}^m$, respectively $\mathbf{w}\in \mathbb{R}^n$, as a vector on (the vertices of) the quotient, respectively the parent, network. Then, each vector $S\mathbf{v}_i$ equals the vector $\mathbf{v}_i$ \emph{lifted} to the parent network by repeating the value on each orbit. Similarly, $S^T\mathbf{w}_j=0$ means that the sum of the entries of $\mathbf{w}_j$ on each orbit is 0. 
All in all, we can always find an eigenbasis of $A$ consisting of \emph{non-redundant} eigenvectors $\{S\mathbf{v}_1,\ldots,S\mathbf{v}_m\}$ arising from a quotient eigenbasis by repeating values on each orbit, and \emph{redundant} eigenvectors $\{\mathbf{w}_1,\ldots,\mathbf{w}_{n-m}\}$ arising from the network symmetries, which add up to zero on each orbit (hence `dissappering' in the quotient). Similarly, we call their respective eigenvalues \emph{redundant} and \emph{non-redundant}.


Analogous to the way that symmetry is generated at symmetric motifs, the redundant eigenvectors and eigenvalues arise directly from certain eigenvectors and eigenvalues of the symmetric motifs, considered as networks on their own (Fig.~\ref{fig:red}). In fact, each symmetric motif $\mathcal M$ contributes the same (called \emph{redundant}) eigenpairs to \emph{any} network containing $\mathcal M$ as a symmetric motif: One can show (Methods, Theorem \ref{thm:redeig}) that if $\mathcal{M}$ is a symmetric motif of a network $\mathcal{G}$ and $(\lambda,\mathbf{w})$ is a redundant eigenpair of $\mathcal M$ (that is, the values of $\mathbf{w}$ add up to zero on each orbit of $\mathcal M$), then $(\lambda, \widetilde{\mathbf{w}})$ is an eigenpair of $\mathcal G$, where $\widetilde{\mathbf{w}}$ is equal to $\mathbf{w}$ on (the vertices of) $\mathcal M$, and zero elsewhere. We call such a vector $\widetilde{\mathbf{w}}$ \emph{localised} on the motif $\mathcal{M}$\cite{macarthur2009spectral}, as it is zero outside the motif. Moreover, if $\mathcal{M}$ has $n$ vertices and $k$ orbits, then it has an eigenbasis consisting of $n-k$ redundant eigenpairs, which are inherited by any network containing $\mathcal{M}$ as a symmetric motif (Fig.~\ref{fig:red}, Theorem \ref{thm:redeig} in Methods).

Furthermore, since most symmetric motifs in real-world networks are basic, thus have a very constrained structure (Fig.~\ref{fig:uniform}), we can in fact determine the redundant spectrum of BSMs with up to a few orbits, that is, we can predict where the most significant `peaks' in the spectral density of an arbitrary network function will occur. The formulae for the redundant spectra for BSMs of one or two orbits (which covers most BSMs, up to 99\% of them in our test networks) is given on Table \ref{table:redundant}. 
For example, for the graph Laplacian, symmetry explains between 89\% and 97\% of the discrete spectrum in our test networks, with most of the high-multiplicity eigenvalues (`peaks' in the eigenvalue histograms) occurring at the predicted values (resulting from our analysis of the spectrum of the most common symmetric motifs, namely BSMs up to 2 orbits, see below and Table \ref{table:redundant}), namely positive integers  (Fig.~\ref{fig:spectral}). 

\begin{figure}
\centering
\begin{subfigure}{0.2\textwidth}
\hspace*{-10pt}\raisebox{-0.6cm}{\begin{tikzpicture}
 	\node at (-0.56,1) {(a)};
	\Vertex[x=0,y=1,size=0.3,color=red,label=1,fontscale=1.1,style=thin]{1}
	\Vertex[x=0,y=0,size=0.3,color=red,label=2,fontscale=1.1,style=thin]{2}
	\Vertex[x=1,y=1,size=0.3,color=blue!70!white,label=3,fontscale=1.1,style=thin]{3}
	\Vertex[x=1,y=0,size=0.3,color=blue!70!white,label=4,fontscale=1.1,style=thin]{4}
	\Vertex[x=2,y=0.5,size=0.3,opacity=0,label=5,fontscale=1.1,style=thin]{5}
	\Edge[lw=0.8pt](1)(3)
	\Edge[lw=0.8pt](2)(4)
	\Edge[lw=0.8pt](3)(5)
	\Edge[lw=0.8pt](4)(5)
\end{tikzpicture}}
\end{subfigure}
\begin{subfigure}[t]{0.25\textwidth}
\hspace{-4pt}\scalebox{0.8}{\renewcommand*{\arraystretch}{1}
$\begin{array}{c|l} 
	\text{Eigenval.} & \text{Eigenvectors}\\ \hline
	1 & (1,-1 \mid 1, -1 \mid 0)\\
	-1 & (1,-1 \mid -1, 1 \mid 0)\\\hline
	\sqrt{3} & (1,1 \mid \sqrt{3}, \sqrt{3} \mid 2)\\
	-\sqrt{3} & (1,1 \mid -\sqrt{3}, -\sqrt{3} \mid 2)\\
	0 & (1,1 \mid 0, 0 \mid -1)\\
\end{array}$}
\end{subfigure}
\begin{subfigure}{0.18\textwidth}
 \hspace*{-10pt}\begin{tikzpicture}
 	 \node at (-0.55,1.2) {(b)};
	\Vertex[x=0,y=1,size=0.3,color=red,label=1,fontscale=1.1,style=thin]{1}
	\Vertex[x=0,y=0,size=0.3,color=red,label=2,fontscale=1.1,style=thin]{2}
	\Vertex[x=1,y=1,size=0.3,color=blue!70!white,label=3,fontscale=1.1,style=thin]{3}
	\Vertex[x=1,y=0,size=0.3,color=blue!70!white,label=4,fontscale=1.1,style=thin]{4}
	\Vertex[x=2,y=0.5,size=0.3,opacity=0,label=5,fontscale=1.1,style=thin]{5}
	\Edge[lw=0.8pt,loopposition=90,label=1.59,fontscale=0.8,position=below](1)(1)
	\Edge[lw=0.8pt,loopposition=-90,label=1.59,fontscale=0.8,position=above](2)(2)
	\Edge[lw=0.8pt,loopposition=90,label=2.29,fontscale=0.8,position=below](3)(3)
	\Edge[lw=0.8pt,loopposition=-90,label=2.29,fontscale=0.8,position=above](4)(4)
	\Edge[lw=0.8pt,loopposition=0,label=3.42,fontscale=0.8,position=left](5)(5)
	\Edge[lw=0.8pt,label=0.05,fontscale=0.8,position=left](1)(2)
	\Edge[lw=0.8pt,label=1.39,fontscale=0.8,position=above](1)(3)
	\Edge[lw=0.8pt](1)(4)
	\Edge[lw=0.8pt,label=0.70,fontscale=0.8,position={above=0pt},distance=0.8,bend=-5,NotInBG](1)(5)
	\Edge[lw=0.8pt](2)(3)
	\Edge[lw=0.8pt,label=1.39,fontscale=0.8,position=below](2)(4)
	\Edge[lw=0.8pt,label=0.70,fontscale=0.8,position={below=-0.5pt},distance=0.8,bend=5,NotInBG](2)(5)
	\Edge[lw=0.8pt,label=0.75,fontscale=0.8,position=right](3)(4)
	\Edge[lw=0.8pt,label=1.87,fontscale=0.8,position=above,bend=35](3)(5)
	\Edge[lw=0.8pt,label=1.87,fontscale=0.8,position=below,bend=-35](4)(5)
	\draw[very thin,->,draw=black!80!white] (-0.3,0.8)  to[out=-45,in=180] (0.25,0.7);
	\node [left] at (-0.2,0.9) [style={scale=0.55}] {$0.21$};
	\draw[very thin,->,draw=black!80!white] (-0.3,0.2)  to[out=45,in=180] (0.25,0.3);
	\node [left] at (-0.2,0.1) [style={scale=0.55}] {$0.21$};
\end{tikzpicture}
\end{subfigure}
\hspace{2pt}
\begin{subfigure}{0.26\textwidth}
\hspace*{1.5pt}\scalebox{0.8}{\renewcommand*{\arraystretch}{1}
$\begin{array}{c|l} 
	\text{Eigenval.} & \text{Eigenvectors}\\ \hline
	e & (1,-1 \mid 1, -1 \mid 0)\\
	1/e & (1,-1 \mid -1, 1 \mid 0)\\\hline
	0.26 & (1,1 \mid -1.38, -1.38 \mid 1.20)\\
	1.28 & (1,1 \mid 0.25, 0.25 \mid -1.09)\\
	6.57 & (1,1 \mid 1.89, 1.89 \mid 2.71)\\
\end{array}$}
\end{subfigure}
\begin{subfigure}{0.2\textwidth}
\hspace*{-22pt}\begin{tikzpicture}
	\node at (-0.54,1) {(c)};
	\Vertex[x=0,y=1,size=0.3,color=red,label=1,fontscale=1.1,style=thin]{1}
	\Vertex[x=0,y=0,size=0.3,color=red,label=2,fontscale=1.1,style=thin]{2}
	\Vertex[x=1,y=1,size=0.3,color=blue!70!white,label=3,fontscale=1.1,style=thin]{3}
	\Vertex[x=1,y=0,size=0.3,color=blue!70!white,label=4,fontscale=1.1,style=thin]{4}
	\Vertex[x=2,y=0.5,size=0.3,opacity=0,label=5,fontscale=1.1,style=thin]{5}
	\Edge[lw=0.8pt,loopposition=90,label=1,Math,fontscale=1,position=below](1)(1)
	\Edge[lw=0.8pt,loopposition=-90,label=1,Math,fontscale=1,position=above](2)(2)
	\Edge[lw=0.8pt,loopposition=90,label=2,Math,fontscale=1,position=below](3)(3)
	\Edge[lw=0.8pt,loopposition=-90,label=2,Math,fontscale=1,position=above](4)(4)
	\Edge[lw=0.8pt,loopposition=0,label=2,Math,fontscale=1,position=left](5)(5)
	\Edge[lw=0.8pt,label=-1,Math,fontscale=1,position=above](1)(3)
	\Edge[lw=0.8pt,label=-1,Math,fontscale=1,position=below](2)(4)
	\Edge[lw=0.8pt,label=-1,Math,fontscale=1,position=above](3)(5)
	\Edge[lw=0.8pt,label=-1,Math,fontscale=1,position=below](4)(5)
\end{tikzpicture}
\end{subfigure}
\begin{subfigure}{0.2\textwidth}
\hspace*{-16pt}\scalebox{0.8}{\renewcommand*{\arraystretch}{0.9}
$\begin{array}{c|l} 
	\text{Eigenval.} & \text{Eigenvectors}\\ \hline
	\varphi+1 & (1,-1 \mid 1, -1 \mid 0)\\
	-\varphi+2 & (1,-1 \mid -1, 1 \mid 0)\\ \hline
	0 & (1,1 \mid 1,1 \mid 1)\\
	\varphi+2 & (1,1 \mid -\varphi-1, -\varphi-1 \mid \eta_1)\\
	-\varphi+3 & (1,1 \mid \varphi-2, \varphi-2 \mid \eta_2)\\
\end{array}$}
\end{subfigure}
\caption{\textbf{Redundant spectrum.} 
\small 
Eigenvalues and eigenvectors near the symmetric motif $M_1$ (Fig.~\ref{fig:toy} (inset)) for three network measures: (a) adjacency $A$, (b) communicability $e^A$, and (c) Laplacian $L$, represented as weighted networks with weights shown as edge labels (as in Fig.~\ref{fig:new}). Eigenvector coordinates are separated by orbit, for convenience. In all three cases, there are two redundant, and three non-redundant, eigenpairs (separated here by a horizontal line), whose eigenvectors add up to zero, respectively are constant, on each orbit. Redundant eigenvectors are localised (zero outside the motif) and `survive' in any network where this graph is a symmetric motif. The redundant eigenvalues and eigenvectors agree with those predicted by the formulae in Table \ref{table:redundant}. Here, $\varphi=\frac{1+\sqrt{5}}{2}$, the golden ratio,
$\eta_1=2+\frac{2}{\varphi}$ and $\eta_2=2+\frac{2}{1-\varphi}$. 
}
\label{fig:red}
\end{figure}

\begin{figure}
\begin{tikzpicture}
\matrix {
\begin{axis}[
	width=0.19\textwidth,
	title=HumanDisease,
	no marks,
	xmin=-0.5,
	xmax=10,	
	ymax=0.25,
	minor x tick num=1,
	]
\addplot table [
	x index = 0, 
	y index = 1,
	] {HumanDisease.hist};
\addplot table [
	x index = 0, 
	y index = 2,
	] {HumanDisease.hist};
\node[fill=white,draw] at (6,0.17) {90.6\%};
\end{axis}
	&
\begin{axis}[
	width=0.19\textwidth,
	title=Yeast,
	no marks,
	xmin=-0.5,
	xmax=10,	
	ymax=0.25,
	minor x tick num=1,
	]
\addplot table [
	x index = 0, 
	y index = 1,
	] {Yeast.hist};
\addplot table [
	x index = 0, 
	y index = 2,
	] {Yeast.hist};
\node[fill=white,draw] at (6,0.17) {94.2\%};
\end{axis}
	&
\begin{axis}[
	title=OpenFlights,
	width=0.19\textwidth,
	no marks,
	xmin=-0.5,
	xmax=10,
	ymax=0.25,	
	minor x tick num=1,
	]
\addplot table [
	x index = 0, 
	y index = 1,
	] {OpenFlights.hist};
\addplot table [
	x index = 0, 
	y index = 2,
	] {OpenFlights.hist};
\node[fill=white,draw] at (6,0.17) {89.2\%};
\end{axis}
	\\[1em]
\begin{axis}[
	title=USPowerGrid,
	width=0.19\textwidth,
	no marks,
	xmin=-0.5,
	xmax=10,	
	ymax=0.25,
	minor x tick num=1,
	]
\addplot table [
	x index = 0, 
	y index = 1,
	] {USPowerGrid.hist};
\addplot table [
	x index = 0, 
	y index = 2,
	] {USPowerGrid.hist};
\node[fill=white,draw] at (6,0.17) {89.9\%};
\end{axis}
	& 
\begin{axis}[
	title=HumanPPI,
	width=0.19\textwidth,
	no marks,
	xmin=-0.5,
	xmax=10,	
	ymax=0.25,
	minor x tick num=1,
	]
\addplot table [
	x index = 0, 
	y index = 1,
	] {HumanPPI.hist};
\addplot table [
	x index = 0, 
	y index = 2,
	] {HumanPPI.hist};
\node[fill=white,draw] at (6,0.17) {95.0\%};
\end{axis}
	&
\begin{axis}[
	title=Astro-Ph,
	width=0.19\textwidth,
	no marks,
	xmin=-0.5,
	xmax=20,	
	ymax=0.25,
	minor x tick num=1,
	]
\addplot table [
	x index = 0, 
	y index = 1,
	] {Astro-Ph.hist};
\addplot table [
	x index = 0, 
	y index = 2,
	] {Astro-Ph.hist};
\node[fill=white,draw] at (12,0.17) {96.7\%};
\end{axis}

	\\
};
\end{tikzpicture}
\caption{\textbf{Spectral signatures of network symmetry.} \small Laplacian spectrum of six test networks (blue) and of their quotient (red), given as relative probability of eigenvalue count, with multiplicity, in bins of size $0.1$. Only the most significant part of the spectrum is shown. Most of the `peaks' observed in the spectral density occur at positive integers, as predicted. (Insets) Percentage of the high-multiplicity spectrum explained by the symmetry, as the ratio of $\sum_{m_\lambda > 1} m_\lambda$ for the quotient eigenvalues, and for the Laplacian eigenvalues, where $m_\lambda$ is the multiplicity of an eigenvalue $\lambda$ rounded to 8 decimal places.
}
\label{fig:spectral}
\end{figure}

We now give more details of the computation of the redundant spectrum of BSMs up to two orbits (Table \ref{table:redundant}), with full details in the Methods section. 
A BSM with one orbit is an $(\alpha,\beta)$-uniform graph $K_n^{\alpha,\beta}$ with adjacency matrix $A_n^{\alpha,\beta}=(a_{ij})$ given by $a_{ij}=\alpha$ and $a_{ii}=\beta$ for all $i\neq j$. Then $K_n^{\alpha,\beta}$ has eigenvalues $(n-1)\alpha+\beta$ (non-redundant), with multiplicity 1, and $-\alpha + \beta$ (redundant), with multiplicity $n-1$. The corresponding eigenvectors are $\mathbf{1}$, the constant vector 1 (non-redundant), and $\mathbf{e}_i$, the vectors with non-zero entries 1 at position 1, and $-1$ at position $i$, $2\le i \le n$ (redundant).  For unweighted graphs without loops ($\beta=0$, $\alpha \in \{0,1\}$), we recover the redundant eigenvalues $0$ and $-1$ predicted in \cite{macarthur2009spectral}. 

A BSM with two orbits must be a uniform join of the form 
$K_n^{\alpha_1,\beta_1} \stackrel{\gamma,\delta}{\longleftrightarrow} K_n^{\alpha_2,\beta_2}$ (Fig.~\ref{fig:uniform}). Let $\kappa_1$ and $\kappa_2$ be the two solutions of the quadratic equation $c \kappa^2 + (b-a)\kappa - c = 0$, where $a=\alpha_1-\beta_1$, $b=\alpha_2-\beta_2$ and $c=\gamma-\delta$. Then, the redundant eigenvalues of this BSM are (Methods, Theorem \ref{thm:RSpec})
\begin{align}
	\lambda_1 &= -b-c\kappa_1 = \frac{-(a+b) + \sqrt{(a-b)^2+4c^2}}{2},  \text{ and}\\
	\lambda_2 &= -b-c\kappa_2 = \frac{-(a+b) - \sqrt{(a-b)^2+4c^2}}{2},
\end{align}
each with multiplicity $n-1$, with eigenvectors $(\kappa_1 \mathbf{e}_i | \mathbf{e}_i)$ and $(\kappa_2 \mathbf{e}_i | \mathbf{e}_i)$ respectively, $2\le i\le n$. For unweighted graphs without loops, we recover the redundant eigenvalues predicted in \cite{macarthur2009spectral}, that is, 
\begin{equation}\label{eqn:golden}
	-2,\ -\varphi,\ -1,\ 0,\ \varphi-1 \text{ and } 1,
\end{equation}
where $\varphi=\frac{1+\sqrt{5}}{2}$, the golden ratio.

\begin{table}
\begin{tabular}{cccc}
BSM & eigenvalues & mult & eigenvectors \\
\hline
$K_n^{\alpha,\beta}$ &  $-\alpha + \beta$ & $n-1$ & $\mathbf{e}_i$\\
\hline
\multirow{2}{*}{$K_n^{\alpha_1,\beta_2} \stackrel{\gamma, \delta}{\longleftrightarrow} K_n^{\alpha_2,\beta_2}$} & $-b - \kappa_1c$ & $n-1$ & $(\kappa_1 \mathbf{e}_i \,|\, \mathbf{e}_i)$\\
& $-b - \kappa_2c$ & $n-1$ & $(\kappa_2 \mathbf{e}_i \,|\, \mathbf{e}_i)$\\
\hline
\end{tabular}
\caption[lof]{\textbf{Redundant spectra of basic symmetric motifs (BSMs) with one or two orbits.} \small 
A BSM with one orbit is a uniform graph $K_n^{\alpha,\beta}$ with $n$ vertices and adjacency matrix $A_n^{\alpha,\beta}=(a_{ij})$ where $a_{ij}=\alpha$ if $i \neq j$ and $a_{ii}=\beta$, for all $i, j$ and some constants $\alpha$ and $\beta$. A BSM with two orbits consists of the $(\gamma,\delta)$-uniform join of two uniform graphs $K_n^{\alpha_1,\beta_1}$ and $K_n^{\alpha_2,\beta_2}$, that is, the graph with $2n$ vertices and block adjacency matrix (after a suitable labelling of the vertices) of the form $\left(\begin{smallmatrix} A & C \\  C & B \end{smallmatrix}\right)$ where $A=A_n^{\alpha_1,\beta_1}$, $B=A_n^{\alpha_2,\beta_2}$ and $C=A_n^{\gamma,\delta}$, each defined as above. 
We write $\mathbf{e}_i$ for the vector with non-zero entries 1 at position 1, and $-1$ at position $i$ ($2\le i \le n$), $\kappa_1$ and $\kappa_2$ for the two solutions of the quadratic equation $c\kappa^2+(-a+b)\kappa-c=0$ where $a=\alpha_1-\beta_1$, $b=\alpha_2-\beta_2$ and $c=\gamma-\delta$, and use $(v|w)$ to represent the concatenation of two vectors. 
}
\label{table:redundant}
\end{table}

\subsection*{Eigendecomposition algorithm}
Decoupling the contribution to the network spectrum from the symmetric motifs and from the quotient network, as explained above, naturally leads to an eigendecomposition algorithm that exploits the presence of symmetries: The spectrum and eigenbasis of an undirected network (equivalently, a diagonalisation of its adjacency matrix $A = UDU^T$) can be obtained from those of the quotient, and of the symmetric motifs, reducing the computational time (cubic on the number of vertices) to up to a third in our test networks (Fig.~\ref{fig:reduction}, left column of the spectral case), in line with our predictions ($sp=n^3_\mathcal{Q}$ in Table \ref{Table1}). The algorithm is shown and explained below. A MATLAB implementation is available at a public repository\cite{BitBucket}.

Our eigendecomposition algorithm (Alg.~\ref{alg:spectraldecomp}) applies to any undirected matrix with symmetries (identifying a matrix with the network it represents). It first computes the eigendecomposition of the quotient matrix, then, for each motif, the redundant eigenpairs. Namely, it first computes the spectral decomposition \texttt{eig} of the symmetric quotient $B_\text{sym}=\Lambda^{-1/2}S^TAS\Lambda^{-1/2}$ where $\Lambda$ is the diagonal matrix of the orbit sizes (which can be obtained as the column sums of $S$). This matrix is symmetric and has the same eigenvalues as the left quotient. Moreover, if $B_\text{sym} = U_q D_q U_q^{-1}$ then the left quotient eigenvectors are the columns of $\Lambda U_q$. These become, in turn, eigenvectors of $A$ by repeating their values on each orbit, and can be obtained mathematically by left multiplying by the characteristic matrix $S$. Then, for each motif, we compute the redundant eigenpairs using a null space matrix (explained below), storing eigenvalues and localised (zero outside the motif) eigenvectors. 

Only redundant eigenvectors of a symmetric motif (that is, those which add up to zero on each orbit) become eigenvectors of $A$ by extending them as zero outside the symmetric motif. Therefore, we need to construct redundant eigenvectors from the ouput of \texttt{eig} on each motif (the spectral decomposition of the corresponding submatrix). If $U_\lambda = \left(\begin{array}{@{}c|c|c@{}} \mathbf{v}_1 & \ldots & \mathbf{v}_k \end{array}\right)$ are $\lambda$-eigenvectors of a symmetric motif with characteristic matrix of the orbit partition $S_\text{sm}$, we need to find linear combinations such that 
\begin{align}
	S_\text{sm}^T \left(\alpha_1 \mathbf{v}_1 +\ldots+ \alpha_k \mathbf{v}_k \right) = \mathbf{0} \iff S_\text{sm}^T U_\lambda \begin{pmatrix} \alpha_1 \\ \vdots \\ \alpha_k. \end{pmatrix}.
\end{align}
Therefore, if the matrix $Z \neq 0$ represents the null space of $S_\text{sm}^T U_\lambda$, that is, $S_\text{sm}^T U_\lambda Z = 0$ and $Z^TZ=0$, then the columns of $U_\lambda Z$ are precisely the redundant eigenvectors. This is implemented in Alg.~\ref{alg:spectraldecomp} within the innermost \textbf{for} loop. 

\begin{algorithm2e}
\SetAlgoLined
\DontPrintSemicolon

\SetKwData{Motif}{motif}\SetKwData{Orbits}{orbits}
\SetKwFunction{Sum}{sum}\SetKwFunction{Diag}{diag}\SetKwFunction{Eig}{eig}\SetKwFunction{Unique}{unique}\SetKwFunction{Null}{null}\SetKwFunction{Ncol}{ncol}


\KwIn{adjacency matrix $A$, characteristic matrix $S$, list of motifs}
\KwOut{spectral decomposition $A = UDU^T$}
\BlankLine
initialise $U$, $D$ to zero matrices\;
$\Lambda \leftarrow \Diag(\Sum(S))$\;
$B_\text{sym} \leftarrow \Lambda^{-1/2}S^TAS\Lambda^{-1/2}$\;
$[U_q,D_q] \leftarrow \Eig(B_\text{sym})$ so that $B_\text{sym} = U_q D_q U_q^{-1}$\;
$U_q \leftarrow \Lambda U_q$\;
$U \leftarrow \left(\begin{array}{@{}c|c@{}} S U_q  & 0  \end{array}\right)$\;
$D \leftarrow \left(\begin{array}{@{}c|c@{}} D_q  & 0\\ \hline 0 & 0  \end{array}\right)$\;

\ForEach{\Motif}{
	$A_\text{sm} \leftarrow A(\Motif,\Motif)$\;
	compute \Orbits from \Motif and $S$\;
	$S_\text{sm} \leftarrow S(\Motif,\Orbits)$\;
	$[U_\text{sm},D_\text{sm}] \leftarrow \Eig(A_\text{sm})$\;
		\For{$\lambda \in \Unique(\Diag(D_\text{sm}))$}{
			$U_\lambda \leftarrow$ $\lambda$-eigenvectors 			from $U_\text{sm}$\;
			$Z \leftarrow \Null(S_\text{sm}^TU_\lambda)$\;
			$d \leftarrow \Ncol(Z)$\;
			\If{$d > 0$}{
				store $U_\lambda Z$ in $U$\;
				store $\lambda$ in $D$ with multiplicity $d$\;
			}
		}
}
\caption{Eigendecomposition algorithm.}
\label{alg:spectraldecomp}
\end{algorithm2e}

\subsection*{Vertex measures} \label{sec:Centrality} 
We have so far considered network measures of the form $F(i,j)$, where $i$ and $j$ are vertices. However, many important network measurements are vertex based, that is, of the form $G(i)$ for each vertex $i$. We say that a vertex measure $G$ is \emph{structural} if it only depends on the network structure and, therefore, satisfies
\begin{equation}\label{eq:vertexfun}
	G(i)=G(\sigma(i))
\end{equation}
for each automorphism $\sigma \in \text{Aut}(\mathcal G)$, that is, it is constant on orbits (Fig.~\ref{fig:toy}).

Although for vertex measures we do not have a network representation, we can still exploit the network symmetries. First, $G$ needs only to be computed/stored once per orbit, resulting on a reduction/compression ratio of $\widetilde{n}_\mathcal{Q}=n_\mathcal{Q}/n_\mathcal{G}$ (Table \ref{Table1}). 
Secondly, when quotient recovery holds (that is, we can recover $G$ from its values on the quotient and symmetry information alone), it amounts to a further computational reduction (Fig.~\ref{fig:reduction}), depending on the computational complexity of $G$. Finally, many vertex measures arise nevertheless from a pairwise function, such as $G(i)=F(i,i)$ (subgraph centrality from communicability), or $G(i)=\frac{1}{n}\sum_{j} F(i,j)$ (closeness centrality from shortest path distance), allowing the symmetry-induced results on $F$ to carry over to $G$. 

\subsection*{Applications}
We illustrate our methods on several popular pairwise and vertex-based network measures. Although novel and of independent interest, these are example applications: Our methods are general and the reader should be able to adapt our results to the network measure of their interest. 

\smallskip
\noindent\textbf{Adjacency matrix.} 
The methods in this paper can be applied to the network itself, that is, to its adjacency matrix. We recover the structural and spectral results in \cite{macarthur2008symmetry,macarthur2009spectral}, and the quotient compression ratio reported in \cite{xiao2008network}, here $c_\text{sparse}=\widetilde{m}_\mathcal{Q}$ in Table \ref{Table1}. The network (adjacency) eigendecomposition can be significantly sped up by exploiting symmetries (Fig.~\ref{fig:reduction}). 

\smallskip
\noindent\textbf{Communicability.}
Communicability is a very general choice of structural measure, consisting on any analytical function $f(x)= \sum a_n x^n$ applied to the adjacency matrix, 
$
	f(A) = \sum_{n=0}^\infty a_n A^n,
$
and it is a natural measure of network connectivity, since the matrix power $A^k$ counts walks of length $k$ \cite{estrada2008communicability}. The most common choice of coefficients is $a_n=\frac1{n!}$, which gives the exponential matrix $e^A=\sum_{n=0}^\infty \frac{A^n}{n!}$.
Communicability is a structural network measure and its network representation, the graph $f(\mathcal G)$ with adjacency matrix $f(A)$, inherits all the symmetries of $\mathcal G$ and thus it has the same symmetric motifs and orbits. The BSMs are uniform joins of orbits, and each orbit is a uniform graph (Figs.~\ref{fig:uniform} and \ref{fig:new}b) characterised by the communicability of a vertex to itself (a natural measure of centrality \cite{estrada2005subgraph}), and the communicability between distinct vertices.
As a full network measure, the compression ratio $c_\text{full}$ applies (Table \ref{Table1}), indicating the fraction of storage needed by using the quotient to eliminate redundancies (Fig.~\ref{fig:compression}). Moreover, average quotient recovery holds for communicability since $f(Q(A))=Q(f(A))$ (Methods, Theorem \ref{thm:comm}). Alternatively, we can use the spectral decomposition algorithm on the adjacency matrix ($A = U D U^T$ implies $f(A) = U f(D) U^T$) reducing the computation, typically cubic on the number of vertices, by $sp = \widetilde{n}_\mathcal{Q}^3 $ (Table \ref{Table1}, Fig.~\ref{fig:reduction}). For the spectral results, note that $f(A)=U f(D) U^T$ has eigenvalues $f(\lambda)$, and same eigenvectors, as $A$. Thus
\begin{align}
	f(-2), f(-\varphi), f(-1), f(0), f(\varphi-1), \text{ and } f(1)
\end{align}
account for most of the discrete part of the spectrum $f(A)$, for the adjacency matrix $A$ of a typical (undirected, unweighted) real-world network (Eq.~\eqref{eqn:golden}).

\smallskip
\noindent\textbf{Shortest path distance.}
This is the simplest metric on a (connected) network, namely the length of a shortest path between vertices. 
A \emph{path} of \emph{length} $n$ is a sequence $(v_1,v_2,\ldots,v_{n+1})$ of distinct vertices, except possibly $v_1=v_{n+1}$, such that $v_i$ is connected to $v_{i+1}$ for all $1\le i \le n-1$. The \emph{shortest path distance}  $d^\mathcal{G}(u,v)$  is the length of the shortest (minimal length) path from $u$ to $v$. If $p=(v_1,v_2,\ldots,v_n)$ is a path and $\sigma \in \text{Aut}(\mathcal G)$, we define $\sigma(p) = (\sigma(v_1),\sigma(v_2),\ldots,\sigma(v_n))$, also a path since $\sigma$ is a bijection. 
A path $p$ is a \emph{shortest path} if it is of minimal length between its endpoints. 
One can show that (i) automorphisms preserve shortest paths and their lengths; (ii) shortest paths between vertices in different symmetric motifs do not contain intra-orbit edges; and (iii) shortest path distance is a partially quotient recoverable structural measure (Methods, Theorem \ref{thm:SPD}). In particular, automorphisms $\sigma$ preserve the shortest path metric, $d(i,j)=d\left(\sigma(i),\sigma(j)\right)$, and we can compute shortest distances from the quotient, 
\begin{equation}\label{eq:shortestpathrecovery}
	d^{\mathcal G}(\alpha,\beta) = d^{\mathcal Q}(i,j), \quad \alpha \in V_i, \beta \in V_j,  
\end{equation}
whenever $V_i$ and $V_j$ are orbits in different symmetric motifs. This accounts for all but the (small) intra-motif distances and reduces the computation as shown in Fig.~\ref{fig:reduction}. 

Distances between points within the same motif cannot in general be directly recovered from the quotient, not even for BSMs. (Consider for instance the double star, motif $M_1$, in Fig.~\ref{fig:toy}: The distance from the top red to the bottom blue vertex is three, while in the quotient is one.) In general, therefore, the shortest path distance is partially, but not average, quotient recoverable. Intra-motif distances, if needed, could still be recovered one motif at a time. 

Note that these results can be exploited for other graph-theoretic notions defined in terms of distance, for example eccentricity (and thus radius or diameter), which only depends on maximal distances and thus it can be computed directly in the quotient.

In terms of symmetry compression, the compression ratio $c_\text{full}$ applies, accounting for the amount of structural redundancy due solely to symmetries. The spectral results, although perhaps less relevant, still apply for $d(\mathcal G)$, the graph encoding pairwise shortest path distances. The adjacency matrix $d(A)=(d^\mathcal{G}(i,j))$ is nonzero outside the diagonal, hence $d(\mathcal G)$ is a all-to-all weighted network without self-loops and integer weights, and so is each symmetric motif. Using the formula in Table \ref{table:redundant}, we can easily compute values of the most significant part of the discrete spectrum (redundant eigenvalues) of $d(A)$,
namely $-3$, $-2$, $-1$, $0$, $-2\pm\sqrt{2}$, $-3\pm\sqrt{2}$, $\frac{-3\pm\sqrt{5}}{2}$, $\frac{-5\pm\sqrt{5}}{2}$ and $\frac{-5\pm\sqrt{13}}{2}$.


\smallskip
\noindent\textbf{Laplacian matrix.}
The Laplacian matrix of a network $L=D-A$, where $D$ is the diagonal matrix of vertex degrees, is a (sparse) network measure and therefore inherits all the symmetries of the network. The matrix $L$ can be seen as the adjacency matrix of a network $\mathcal L$ with identical symmetric motifs, except that all edges are weighted by $-1$ and all vertices have self-loops weighted by their degrees in $\mathcal G$ (Fig.~\ref{fig:new}c). In particular, the motif structure (namely, the self-loop weights) depends on the how the motif is embedded in the network $\mathcal G$. 

Quotient compression and computational reduction are less useful in this case, however the spectral results are more interesting. The spectral decomposition applies, and we can compute redundant Laplacian eigenvalues directly from Table \ref{table:redundant}, for instance positive integers for BSMs with one orbit (Methods, Corollary \ref{cor:laplacian}). This explains and predicts most of the `peaks' (high multiplicity eigenvalues) in the Laplacian spectral density, confirmed on our test networks (Fig.~\ref{fig:spectral}).  
Using the formula in Table \ref{table:redundant}, one can similarly compute the redundant spectrum for 2-orbit BSMs, and for other versions of the Laplacian (e.g.~normalised, vertex weighted). Finally, observe that the spectral decomposition applies, thus Algorithm \ref{alg:spectraldecomp} provides an efficient way of computing the Laplacian eigendecomposition with an expected $sp=\tilde{n}_\mathcal{Q}^3$ (see Table \ref{Table1}) computational time reduction.

\smallskip
\noindent\textbf{Commute distance and matrix inversion.}
The commute distance is the expected time for a random walker to travel between two vertices and back \cite{lovasz1993random}. In contrast to the shortest path distance, it is a global metric which takes into account all possible paths between two vertices. The commute distance is equal up to a constant (the volume of the network) to the resistance metric $r$ \cite{klein1993resistance}, which can be expressed in terms of $L^\dagger = (l^\dagger_{ij})$, the pseudoinverse (or Moore-Penrose inverse) of the Laplacian, as
$r(i,j) = l^\dagger_{ii} + l^\dagger_{jj} - 2l^\dagger_{ij}$.
The commute (or resistance) distance is a (full) structural measure, and all our structural and spectral results apply. Crucially, we can use eigendecomposition algorithm to obtain $L=UDU^T$ (and hence $L^\dagger=UD^\dagger U^T$, and $r$) from the quotient and symmetric motifs, resulting in significant computational gains (Fig.~\ref{fig:reduction}). More generally, if $M_F$ is the matrix representation of a network measure, its pseudoinverse $M_F^\dagger$ is also a network measure, and the comments above apply. Note that $M_F^\dagger$ is generally a full measure even if $M_F$ is sparse (the inverse of a sparse matrix is not generally sparse).

\smallskip
\noindent\textbf{Vertex symmetry compression.}
As a vertex measure $G$ is constant on orbits, we only need to store one value per orbit. Let $S$ be the characteristic matrix of the partition of the vertex set into orbits, and $\Lambda$ the diagonal matrix of orbit sizes (column sums of $S$). If $G$ is represented by a vector $v=(G(i))$ of length $n_\mathcal{G}$, we can store one value per orbit by taking $\mathbf{w}=\Lambda^{-1}S^T\mathbf{v}$, a vector of length $n_\mathcal{Q}$, and recover $\mathbf{v}=S^T\mathbf{w}$ (Methods, Theorem \ref{thm:vertexcompression}).

\smallskip
\noindent\textbf{Degree Centrality.}
The degree of a node (in- or out-degree if the network is directed) is a natural measure of vertex centrality. As expected, the degree is preserved by any automorphism $\sigma$, which can also be checked directly, 
\begin{equation}\label{eq:degree}
	d_i = \sum_{j \in V} a_{ij} = \sum_{j \in V} a_{\sigma(i)\sigma(j)} = \sum_{j \in V} a_{\sigma(i)j} = d_{\sigma(i)},
\end{equation}
as automorphisms permute orbits (so $j \in V$ and $\sigma(j) \in V$ are the same elements but in a different order). In particular, the degree is constants on orbits. We recover the degree centrality from the quotient as the out-degree (Methods, Proposition \ref{prop:degreequotient}). 


\smallskip
\noindent\textbf{Closeness centrality.} 
The closeness centrality of a node $i$ in a graph $\mathcal G$, $cc^\mathcal{G}(i)$, is the average shortest path length to every node in the graph. As symmetries preserve distance, they also preserve closeness centrality, explicitly,
\begin{align}
	cc(i) &= \frac{1}{n_\mathcal{G}} \sum_{j \in V} d(i,j) = \frac{1}{n_\mathcal{G}} \sum_{j \in V} d(\sigma(i),\sigma(j))\nonumber\\ &= \frac{1}{n_\mathcal{G}}  \sum_{j \in V} d(\sigma(i),j) = cc(\sigma(i))\,,
\end{align}
and centrality is constant on each orbit, as expected. Moreover, closeness centrality can be  recovered from the quotient (shortest paths does not contain intra-orbit edges, except between vertices in the same symmetric motif, see above), as
\begin{align}
	cc^{\mathcal G}(i)= \sum_{l\neq k}\frac{n_l}{n_\mathcal{G}} d^{\mathcal Q}(k,l) + \frac{n_i}{n_\mathcal{G}} d_k
\end{align}
if $i$ belongs to the orbit $V_k$ and $d_k$ is the average intra-motif distance, that is, the average distances of a vertex in $V_k$ to any vertex in $\mathcal M$, the motif containing $V_k$. By annotating each orbit by $d_k$, we can recover betweenness centrality exactly. Alternatively, as $d_k \ll n$ (note that $d_k \le m$ if $\mathcal M$ has $m$ orbits), we can approximate $cc^{\mathcal G}(i)$ by the first summand, or simply by the quotient centrality $cc^\mathcal{Q}(\alpha)$, in most practical situations. 

\smallskip
\noindent\textbf{Betweenness centrality.}
This is the sum of proportions of shortest paths between pairs of vertices containing a given vertex. 
It can be computed from shortest path distances and number of shortest paths\cite{brandes2001faster}, both pairwise structural measures, reducing the computation of a naive $O(n^3)$ time, $O(n^2)$ space implementation by $\widetilde{n}^3_\mathcal{Q}$ and $\widetilde{n}^2_\mathcal{Q}$. It would be interesting to adapt a faster algorithm e.g.\cite{brandes2001faster} to exploit symmetries, but this is beyond our scope.

\smallskip
\noindent\textbf{Eigenvector centrality.}
Eigenvector centrality is obtained from a Perron-Frobenius eigenvector (i.e.~of the largest eigenvalue) of the adjacency matrix of a connected graph \cite{newman2010networks}. 
Since this eigenvalue must be simple, it cannot be a redundant eigenvalue. Hence it is a quotient eigenvalue,  and, as those are a subset of the parent eigenvalues, it must still be the largest (hence the Perron-Frobenius) eigenvalue of the quotient. Its eigenvector can then be lifted to the parent network, by repeating entries on orbits. That is, if $(\lambda, v)$ is the Perron-Frobenius eigenpair of the quotient, then $(\lambda,Sv)$ is the Perron-Frobenius eigenpair of the parent network. In practice, we use the symmetric quotient $B_\text{sym} = \Lambda^{-1/2}S^TAS\Lambda^{-1/2}$ for numerical reasons (Algorithm \ref{alg:EC}). Hence the computation (quadratic time by power iteration) can be reduced by $\widetilde{n}^2_\mathcal{Q}$ (Fig.~\ref{fig:reduction}).

\begin{algorithm2e}
\SetAlgoLined
\DontPrintSemicolon
\SetKwFunction{Sum}{sum}\SetKwFunction{Diag}{diag}\SetKwFunction{Eig}{eig}

\KwIn{adjacency matrix $A$, characteristic matrix $S$}
\KwOut{(right) Perron-Frobenius eigenpair $(\lambda,v)$ of $A$}
\BlankLine
$\Lambda \leftarrow \Diag(\Sum(S))$\;
$R \leftarrow S \Lambda^{-1/2}$\;
$B_\text{sym} \leftarrow R^TAR$\;
$(\lambda,w) \leftarrow \Eig(B_\text{sym},1)$ eigenpair of the largest eigenvalue \;
$v \leftarrow R w$\;

\caption{Eigenvector centrality from the quotient network.}
\label{alg:EC}
\end{algorithm2e}

\section*{Discussion}
We have presented a general theory to describe and quantify the effects of network symmetry on arbitrary network measures, and explained how this can be exploited in practice in a number of ways. 

Network symmetry of the large but sparse graphs typically found in applications can be effectively computed and manipulated, making it an inexpensive pre-processing step. 
We showed that the amount of symmetry is amplified in a pairwise network measure but can be easily discounted using the quotient network. We can for instance eliminate the symmetry-induced redundancies, or use them to simplify the calculation by avoiding unnecessary computations. Symmetry has also a profound effect on the spectrum, explaining the characteristic `peaks' observed in the spectral densities of empirical networks, and occurring at values we are able to predict. 

Our framework is very general and apply to any pairwise or vertex-based network measure beyond the ones we discuss as examples. We emphasised practical and algorithmic aspects throughout, and provide pseudocode and full implementations\cite{BitBucket}. Since real-world network models and data are very common, and typically contain a large degree of structural redundancy, our results should be relevant to any network practitioner.

\section*{Methods}
\footnotesize

\subsection*{Geometric decomposition and symmetric motifs}
We write $\text{Aut}(\mathcal G)$ for the automorphism group of an (unweighted, undirected, possibly very large) network $\mathcal G = (V,E)$ (see below for a discussion of directed and weighted networks). 
Each automorphism (symmetry) $\sigma \in \text{Aut}(\mathcal G)$ is a permutation of the vertices and its \emph{support} is the set of vertices moved by $\sigma$,
\begin{equation}
	\textup{supp}(\sigma) =\{ i \in V \text{ such that } \sigma(i) \neq i \}.
\end{equation}
Two automorphisms $\sigma$ and $\tau$ are \emph{support-disjoint} if the intersection of their supports is empty, $\textup{supp}(\sigma) \cap \textup{supp}(\tau) = \emptyset$. The \emph{orbit} of a vertex $i$ is the set of vertices to which $i$ can be moved to by an automorphism, that is,
\begin{equation}\label{eq:orbit}
	\{ \sigma(i) \text{ such that } \sigma \in \textup{Aut}(\mathcal G) \}.
\end{equation}
One can show\cite{macarthur2008symmetry} that there is a partition a set $X$ of generators of $\text{Aut}(\mathcal G)$ into its finest support-disjoint classes $X=X_1\cup\ldots\cup X_m$ which is unique up to permutation of the sets $X_i$. The vertex sets $M_i = \cup_{\sigma \in X_i} \textup{supp}(\sigma)$ give the geometric decomposition Eq.~\eqref{eqn:VertexDecomp}, and the subgraphs induced by them are, by definition, the \emph{symmetric motifs} of $\mathcal G$. (The next section explains how to compute the geometric decomposition in practice.) Since support-disjoint automorphisms must commute (the order in which they are composed is irrelevant), the subgroups of $\text{Aut}(\mathcal G)$ generated by $X_1$ to $X_m$, call them $H_1$ to $H_m$, give a direct product decomposition $\text{Aut}(\mathcal G) = H_1 \times \ldots \times H_m$. The geometric decomposition is defined from the finest support-disjoint partition of a special set of generators (called \emph{essential}), as explained in\cite{macarthur2008symmetry}. However, the results in this article are valid for any support-disjoint decomposition of any set of generators (essential or not) of $\text{Aut}(\mathcal G)$. 

If all the orbits of a symmetric motif have the same size $k$ and every permutation of the vertices in each orbit can be extended to a network automorphism supported on the motif, we call the symmetric motif \emph{basic} (or BSM) of \emph{type} $k$. (In particular, the corresponding subgroup $H_i$ must be $\text{Sym}(k)$, the symmetric group of all permutations of $k$ elements.) If a symmetric motif is not basic, we call it \emph{complex} or of \emph{type} 0.

\subsection*{Network symmetry computation}
First, we compute a list of generators of the automorphism group from an edge list (we use \texttt{saucy3} \cite{saucy3}, which is extremely fast for the large but sparse networks typically found in applications). Then, we partition the set of generators $X$ into support-disjoint classes $X=X_1 \cup \ldots \cup X_m$, that is, $\sigma$ and $\tau$ are support-disjoint whenever $\sigma \in X_i$, $\tau \in X_j$ and $i \neq j$. To find the finest such partition, we use a bipartite graph representation of vertices $V$ and generators $X$. Namely, let $\mathcal B$ be the graph with vertex set $V \cup X$ and edges between $i$ and $\sigma$ whenever $i \in \text{supp}(\sigma)$. Then $X_1, \ldots, X_m$ are the connected components of $\mathcal B$ (as vertex sets intersected with $X$). Each $X_i$ corresponds to the vertex set $M_i$ of a symmetric motif $\mathcal M_i$, as 
$
	M_i = \bigcup_{\sigma \in X_i} \textup{supp}(\sigma).
$
Finally, we use GAP \cite{GAP4} to compute the orbits and type of each symmetric motif (Alg.~\ref{alg:GAP}). Full implementations of all the procedures outlined above are available at a public repository \cite{BitBucket}. 


\begin{algorithm2e}
\SetAlgoLined
\DontPrintSemicolon

\SetKwData{orbit}{orbit}
\SetKwFunction{Group}{Group}\SetKwFunction{Orbits}{Orbits}\SetKwFunction{Size}{size}\SetKwFunction{Max}{max}\SetKwFunction{Min}{min}\SetKwFunction{IsNaturalSymmetricGroup}{IsNaturalSymmetricGroup}\SetKwFunction{Action}{Action}

\KwIn{$X$ a set of permutations of a symmetric motif}
\KwOut{$O_1,\ldots,O_k$ orbits, and type $m$, of the symmetric motif}
\BlankLine
$H \leftarrow \Group(X)$\;
$\{O_1,\ldots,O_k\} \leftarrow \Orbits(H)$\;
\BlankLine
$m \leftarrow \Min(\Size(O_1),\ldots,\Size(O_k))$\;
\eIf{$m ==  \Max(\Size(O_1),\ldots,\Size(O_k))$}{
\For{$i \leftarrow 1$ \KwTo $k$}{
	\If{not $\IsNaturalSymmetricGroup(\Action(H,O_i))$}{
	$m \leftarrow 0$\; break\;}
}
}{$m \leftarrow 0$\;}

\caption{Orbits and type of a symmetric motif.}
\label{alg:GAP}
\end{algorithm2e}

\subsection*{Structural network measures}
We prove below the structural result for BSMs for arbitrary graphs and network measures. The proof is a generalisation of the argument on \cite[p.48]{liebeck1988graphs} to weighted directed graphs with symmetries. 
\begin{theorem}\label{thm:BSM}
Let $M$ be the vertex set of a BSM of a network $\mathcal G$, and $F$ a structural network measure. Then the graph induced by $M$ in $F(\mathcal G)$ is a BSM of $F(\mathcal G)$, and satisfies:
\begin{itemize}
\item[(i)] for each orbit $\Delta=\{v_1,\ldots,v_n\}$, there are constants $\alpha$ and $\beta$ such that the orbit internal connectivity is given by $\alpha=F(v_i,v_j)$ for all $i \neq j$ and $\beta = F(v_i,v_i)$ for all $i$;
\item[(ii)] for every pair of orbits $\Delta_1$ and $\Delta_2$, there is a labelling $\Delta_1=\{v_1,\ldots,v_n\}$, $\Delta_2=\{w_1,\ldots,w_n\}$ and constants $\gamma_1$, $\gamma_2$, $\delta_1$, $\delta_2$ such that $\gamma_1 = F(v_i,w_j)$, $\gamma_2 = F(w_j,v_i)$, $\delta_1 =F(v_i,w_i)$, and $\delta_2=F(w_i,v_i)$, for all $i \neq j$; 
\item[(iii)] every vertex $v$ not in the BSM is joined uniformly to all the vertices in each orbit $\{v_1,\ldots,v_n\}$ in the BSM, that is, $F(v,v_i)=F(v,v_j)$ and $F(v_i,v)=F(v_j,v)$ for all $i,j$. 
\end{itemize}
Moreover, property (iii) holds in general for any symmetric motif.
\end{theorem}
If $\mathcal G$ is undirected and $F$ is symmetric, $\gamma_1=\gamma_2$ and $\delta_1=\delta_2$ and each orbit is a $(\alpha,\beta)$-uniform graph $K_n^{\alpha,\beta}$ and each pair of orbits form a $(\gamma,\delta)$-uniform join, explaining Figure \ref{fig:uniform}(a, b). 
\begin{proof}[Proof of Theorem]
As $F(\mathcal G)$ inherits all the symmetries of $\mathcal G$, $M$ has the same orbit decomposition and the symmetric group $S_n$ acts in the same way, hence $M$ induces a BSM in $F(\mathcal G)$ too. For the internal connectivity, note that every permutation of the vertices $v_i$ is realisable. Thus, given arbitrary $1\le i,j,k,l \le n$, we can find $\sigma \in Aut(\mathcal G)$ such that $\sigma(v_k)=v_i$ and, if $j\neq i$ and $l \neq k$, additionally satisfies $\sigma(v_l)=v_j$. This gives
\[
	F(v_i,v_j)=F(\sigma(v_k),\sigma(v_l))=F(v_k,v_l),
\]
as $F$ is a structural network measure. The other case, $i=j$ and $k=l$, gives
\[
	F(v_i,v_i)=F(\sigma(v_k),\sigma(v_k))=F(v_k,v_k).
\]
For the orbit connectivity result (ii), we generalise the argument in \cite[p.48]{liebeck1988graphs} to weighted directed graphs with symmetries, particularly $F(\mathcal G$). We assume some basic knowledge and terminology about group actions \cite{rotman2012introduction} and symmetric groups $S_n$. Given two orbits $\Delta_1=\{ v_1, \ldots, v_n\}$ and $\Delta_2=\{ w_1, \ldots, w_n\}$ and $1 \le i \le n$, define
\[
	\Gamma_i = \left\{ w_j \in \Delta_2 \,|\, F(v_i,w_j) \neq 0 \right\},
\]
the vertices in $\Delta_2$ joined to $v_i$ in $F(\mathcal G)$. If a finite group $G$ acts on a set $X$, the stabiliser of a point $G_x = \{ g\in G \,|\, gx=x\}$ is a subgroup of $G$ of index $[G : H] =\frac{|G|}{|H|}$ equals to the size of the orbit of $x$. Hence the stabilisers $G_{v_i}$ or $G_{w_j}$ are subgroups of $S_n$ of index $n$, for all $i, j$. The group $S_n$ has a unique, up to conjugation, subgroup of index $n$ if $n\neq 6$. In this case, $G_{v_1}$ is conjugate to $G_{w_1}$ so $G_{v_1} = \sigma G_{w_1} \sigma^{-1} = G_{\sigma w_1}$ for some $\sigma \in S_n$. Relabelling $\sigma w_1$ as $w_1$ we have $G_{v_1}=G_{w_1}$. Similarly, we can relabel the remaining vertices in $\Delta_2$ so that $G_{v_i}=G_{w_i}$ for all $i$: write $v_2=\sigma_2v_1$, $v_3=\sigma_3v_1$, $\ldots$ and relabel $w_2=\sigma_2w_1$, $w_3=\sigma_3w_1$, $\ldots$, noticing there cannot be repetitions as $\sigma_kw_1=\sigma_lw_1$ for $k\neq l$ implies $\sigma_k\sigma_l^{-1} \in G_{w_1}=G_{v_1}$ and thus $v_k=\sigma_kv_1=\sigma_lv_1=v_l$, a contradiction. Fix $1 \le i \le n$. The stabiliser $G_{v_i}$ fixes $v_i$ but it may permute vertices in $\Delta_2$. In fact, the set $\Gamma_i$ above must be a union of orbits of $G_{v_i}$ on $\Delta_j$: if $w \in \Gamma_i$ and $\sigma \in G_{v_i}$ then 
\[
	0 \neq F(v_i,w) = F(\sigma v_i, \sigma w) = F(v_i, \sigma w)
\]
so $\sigma w$ also belongs to $\Gamma_i$.  The orbits of $G_{v_i}=G_{w_i}$ in $\Delta_2$ are $\{w_i\}$ and $\Delta_2 \setminus \{w_i\}$, as $G_{w_i}$ fixes $w_i$ and freely permutes all other vertices in $\Delta_2$. The case $n=6$ is similar, except that $S_6$ has two conjugacy classes of subgroups of index $6$, one as above, and the other a subgroup acting transitively on the 6 vertices, which gives a unique orbit $\Delta_2$. In all cases, the set $\Delta_2 \setminus \{w_i\}$ is part of an $G_{v_i}$-orbit, which gives the connectivity result, as follows. Fix $1 \le i \le n$. For $1 \le j, k \le n$ different from $i$, the vertices $w_j$ and $w_k$ are in the same $G_{v_i}$-orbit so there is $\sigma \in G_{v_i}$ with $\sigma w_j = w_k$ and, therefore,
\[
	F(v_i,w_j) = F(\sigma v_i, \sigma w_j) = F(v_i, w_k).
\]
The argument is general, so we have shown $a_i = F(v_i,w_j)$ is constant for all $j \neq i$. It is enough to show $a_i = a_1$ for all $i$. Choose $j \neq i$, then
\[
	a_i = F(v_i,w_j) = F(\sigma_i v_1, \sigma_j w_1) = F(v_1, \sigma_i^{-1}\sigma_j w_1) = a_1
\]
as long as $ \sigma_i^{-1}\sigma_j w_1 \neq w_1$, which cannot happen as otherwise $ \sigma_i^{-1}\sigma_j \in G_{w_1}=G_{v_1}$ implies $ \sigma_i^{-1}\sigma_j v_1 = v_1$ or $ v_j = \sigma_j v_1 =  \sigma_i v_1 = v_i$, a contradiction. Hence we have shown $F(v_i,w_j)$ is a constant, call it $\gamma_1$, for all $i \neq j$. In addition, 
\[
	F(v_i,w_i)=F(\sigma_i v_1, \sigma_i w_1)=F(v_1,w_1)
\]
is also a constant, call it $\delta_1$, for all $i$. The cases $\gamma_2=F(w_j,v_i)$ and $\delta_2=F(w_i,v_i)$ are identical, reversing the roles of $\Delta_1$ and $\Delta_2$.

Property (iii) holds for any symmetric motif, not necessarily basic, as follows. By the definition of orbit, for each pair $i,j$ we can find an automorphism $\sigma$ in the geometric factor such that $\sigma(v_j)=v_i$. Since $v$ is not in the support of that geometric factor, it is fixed by $\sigma$, that is, $\sigma(v)=v$. Therefore
\[
	F(v,v_i)=F(\sigma(v),\sigma(v_j)) = F(v,v_j),
\]
and similarly $F(v_i,v)=F(v_j,v)$.
\end{proof}

\subsection*{Average compression}
\smallskip
\begin{theorem}\label{thm:avgcomp}
Let $A=(a_{ij})$ be the $n \times n$ adjacency matrix of a network with vertex set $V$. Let $S$ be the $n \times m$ characteristic matrix of the partition of $V$ into orbits of the automorphism group, and $\Lambda$ the diagonal matrix of column sums of $S$. Define $B=S^TAS$ and $A_\textup{avg} = R BR^T = (\bar{a}_{ij})$ where $R=S \Lambda^{-1}$. Then,
\begin{itemize}
\item[(i)] if $i, j \in V$ belong to different symmetric motifs, $\bar{a}_{ij} = a_{ij}$;
\item[(ii)] if $i, j \in V$ belong to orbits $i \in \Delta_1$ and $j \in \Delta_2$ in the same symmetric motif, 
\begin{equation}
	\bar{a}_{ij} = \frac{1}{|\Delta_1|}\frac{1}{|\Delta_2|} \sum_{\substack{u \in \Delta_1\\v \in \Delta_2}} a_{uv}.
\end{equation}
\end{itemize}
\end{theorem}
Before proving this statement, we make a few observations. The column sums of $S$ equal the sizes of the vertex partition sets, hence $\Lambda$ is the same as in the definition of quotient matrix \eqref{eq:QuotientMatrixDef}, and can be obtained easily from $S$. The matrix $S$ is very sparse (each row has a unique non-zero entry) and can be stored very efficiently. Case (i) covers the vast majority of vertex pairs (external edges) for a network measure (see $ext_s$ and $int_f$ in Table \ref{Table1}). In (ii), the case $\Delta_1=\Delta_2$ is allowed. The matrix $B=S^TAS$ is symmetric with integer entries if $A$ is too, hence generally easier to store than $Q(A)=\Lambda^{-1}S^TAS$. 
\begin{proof}[Proof of Theorem]
Let $V=\Delta_1\cup \ldots \cup \Delta_m$ be the partition into orbits, and write $n_k = |\Delta_k|$. Clearly, the row sums of $S$ equals $n_1,\ldots,n_m$. Writing $[M]_{ij}$ for the $(i,j)$-entry of a matrix $M$, matrix multiplication gives
\[
	[R]_{ik} = \sum_{l} [S]_{il} [\Lambda^{-1}]_{lk} \stackrel{l=k}{=} [S]_{ik} \frac{1}{n_k} = 
	\begin{cases} \frac{1}{n_k} & \text{if $i \in \Delta_k$,} \\ 0 & \text{otherwise.}\end{cases}
\]
Similarly, assuming $i \in \Delta_k$ and $j \in \Delta_l$, we have
\[
	\bar{a}_{ij} = [RBR^T]_{ij} = \sum_{\alpha,\beta} [R]_{i \alpha}[B]_{\alpha \beta}[R]_{j\beta} = \frac{1}{n_k}\frac{1}{n_l}[B]_{kl} = \frac{1}{n_k}\frac{1}{n_l} \sum_{\substack{u \in \Delta_k\\ v \in \Delta_l}} a_{uv}.
\]
This expression reduces to $a_{ij}$ if the orbits belong to different symmetric motifs, since in this case all the summands in $\sum_{\substack{u \in \Delta_k,  v \in \Delta_l}} a_{uv}$ are equal to one another. Indeed, given $i_1,i_2 \in \Delta_k$ and $j_1,j_2\in\Delta_l$, we can find, by the definition of orbit and symmetric motif, automorphisms $\sigma$ and $\tau$ such that $\sigma(i_1)=i_2$ while fixing $j_1$, and $\tau(j_1)=j_2$ while fixing $i_1$. This gives
\[
	a_{i_1j_1} = a_{\tau\sigma(i_1)\tau\sigma(j_1)} = a_{\tau(i_2)\tau(j_1)} = a_{i_2j_2}.\qedhere
\]
\end{proof}
A similar proof shows that we can recover exact inter-motif connectivity (external edges), and average intra-motif connectivity (average internal edges) from the quotient network, as follows. 
\begin{corollary}\label{thm:quotientrecovery}
Let $A=(a_{ij})$ be the $n \times n$ adjacency matrix of a network and $Q(A)=(b_{kl})$ its quotient with respect to the partition into orbits of the automorphism group $V=\Delta_1\cup\ldots \cup\Delta_m$. Suppose that $i \in \Delta_k$, $j \in \Delta_l$. Then,
\begin{itemize}
\item[(i)] if the orbits $\Delta_k$ and $\Delta_l$ belong to different symmetric motifs, $b_{kl} = a_{ij}$;
\item[(ii)] if the orbits $\Delta_k$ and $\Delta_l$ belong to the same symmetric motif, 
\begin{equation}
	b_{kl} = \frac{1}{|\Delta_k|} \sum_{\substack{u \in \Delta_k\\v \in \Delta_l}} a_{uv}.
\end{equation}
\end{itemize}
\end{corollary}

\subsection*{Lossless compression}
We can achieve lossless compression by exploiting the structure of BSMs, which account for most of the symmetry in real-world networks. If the motif is basic, we can preserve the exact parent network connectivity in an annotated quotient, as follows. Each orbit in a BSM is a uniform graph $K_n^{\alpha,\beta}$ which appears in the quotient as a single vertex with a self-loop weighted by $(n-1)\alpha+\beta$ (Fig.~\ref{fig:uniform} (c)). Hence if we annotate this vertex in the quotient by not only $n$ but also $\alpha$, or $\beta$, we can recover the internal connectivity. Similarly, the connectivity between two orbits in the same symmetric motif is given by two parameters $\gamma$, $\delta$ and appears in the quotient as an edge weighted $(n-1)\gamma+\delta$ (Fig.~\ref{fig:uniform} (d)) and thus can also be recovered from a quotient with edges annotated by $\gamma$, or $\delta$. 

Since there is no general formula for an arbitrary non-basic symmetric motif, we can work with the \emph{basic quotient} $\mathcal Q_\text{basic}$ instead, that is, the quotient with respect to the partition of the vertex set into orbits in BSMs only (vertices in non-basic symmetric motifs become fixed points hence part of the asymmetric core). The annotated (as above) basic quotient achieves most of the symmetry reduction in a typical empirical network ($\tilde{n}_\mathcal{Q}^\text{basic} \approx \tilde{n}_{\mathcal{Q}}$, $\tilde{m}_\mathcal{Q}^\text{basic} \approx \tilde{m}_{\mathcal{Q}}$, (see Table \ref{Table1}, caption). To maintain the same vertex labelling as in the parent network, we record, for each pairs of orbits in the same symmetric motif, the corresponding permutation of the second orbit (else we recover the adjacency matrix only up to permutations of the orbits). 

Algorithms for lossless compression and recovery based on the basic quotient are shown below (Algorithms \ref{alg:llcompression} and \ref{alg:lldecompression}), and MATLAB implementations for BSMs up to two orbits are available at a public repository\cite{BitBucket}. The results reported in Fig.~\ref{fig:compression} are with respect to these implementations, and the actual compression ratios reported include the size of the annotation data for lossless compression with vertex identity (a very small fraction of the size of the quotient in practice, adding at most $0.02$\% to the basic full compression ratio in all our test cases). 


\begin{algorithm2e}
\SetAlgoLined
\DontPrintSemicolon
\SetKwData{Motifs}{motifs}\SetKwData{MotifType}{motiftype}\SetKwData{Orbits}{orbits}\SetKwData{Rep}{rep}\SetKwData{Orb}{orb}\SetKwData{BSM}{BSM}\SetKwData{bsm}{bsm}\SetKwData{Perm}{perm}
\SetKwFunction{Min}{min}\SetKwFunction{Max}{max}\SetKwFunction{Size}{size}

\KwIn{adjacency matrix $A$, characteristic matrix for the basic quotient $S$, list of BSMs \Motifs}
\KwOut{quotient matrix $B$, annotation structure $a$}
\BlankLine
$B \leftarrow S^TAS$\;
\BlankLine
extract \Orbits from $S$\;
\ForEach{\Orb in \Orbits}{
$\Rep \leftarrow \Min(\Orb)$\;
$\beta \leftarrow A(\Rep,\Rep)$\;
store $\beta$ in annotation structure $a$\;
}
\BlankLine
$k_\text{max} \leftarrow \Max(\Size(\Motifs))$ maximal number of orbits in a motif\;
\For{$k \leftarrow 2$ \KwTo $k_\text{max}$}{
extract $k$-\BSM (list of BSMs with $k$ orbits) from \Motifs\;
\ForEach{\bsm in $k$-\BSM}{
	\ForEach{pairs of distinct orbits $V_1, V_2$ in \bsm}{
		compute $\delta$ and permutation of $V_2$ \Perm such that $A(k,\Perm(k)) = \delta$ for all $k \in V_1$\;
		store orbit numbers (with respect to $S$), $\delta$ and \Perm in annotation structure $a$\;
		}
}
}
\caption{Lossless symmetry compression.} 
\label{alg:llcompression}
\end{algorithm2e}

\begin{algorithm2e}[t]
\SetAlgoLined
\DontPrintSemicolon
\SetKwData{Motifs}{motifs}\SetKwData{MotifType}{motiftype}\SetKwData{Orbits}{orbits}\SetKwData{Rep}{rep}\SetKwData{Orb}{orb}\SetKwData{BSM}{BSM}\SetKwData{bsm}{bsm}\SetKwData{Perm}{perm}\SetKwData{Pairs}{pairs}
\SetKwFunction{Min}{min}\SetKwFunction{Max}{max}\SetKwFunction{Size}{size}

\KwIn{quotient matrix $B$, characteristic matrix $S$, annotation structure $a$}
\KwOut{adjacency matrix $A$}
\BlankLine
$\Lambda \leftarrow \Diag(\Sum(S))$\;
$R \leftarrow S\Lambda^{-1}$\;
$A \leftarrow RBR^T$\;

\BlankLine
extract \Orbits from $S$\;
\ForEach{\Orb in \Orbits}{
$n \leftarrow \Size(\Orb)$\;
extract $\beta$ from $a$\;
compute $\alpha$ from $B$, $\beta$ and $n$ (using $[B]_{\Orb,\Orb} = n\left((n-1)\alpha+\beta\right))$\;
construct adjacency matrix of the orbit $A_n^{\alpha,\beta}$\;
$A(\Orb,\Orb) \leftarrow A_n^{\alpha,\beta}$\;
}

\BlankLine
extract \Pairs of orbits in the same BSM from $a$\;
\ForEach{$(V_1,V_2)$ in \Pairs}{
$n \leftarrow \Size(V_1)$\;
extract $\delta$, \Perm from $a$\;
compute $\gamma$ from $B$, $\delta$ and $n$ (using $[B]_{V_1,V_2} = n\left((n-1)\gamma+\delta\right))$\;
construct matrix $A_n^{\gamma,\delta}$\;
$A(V_1,\Perm) \leftarrow A_n^{\gamma,\delta}$\;
$A(\Perm,V_1) \leftarrow A_n^{\gamma,\delta}$\;
}
\caption{Lossless symmetry decompression.} 
\label{alg:lldecompression}
\end{algorithm2e}


\subsection*{Spectral signatures of symmetry}

The partition into orbits satisfy the following regularity condition\cite{cvetkovic2010introduction, brouwer2011spectra}. A partition of the vertex set $V=V_1\cup \ldots \cup V_m$ is \emph{equitable} if
\begin{equation}
	\sum_{j \in V_l} a_{i_1j} = \sum_{j \in V_l} a_{i_2j} \quad 
	\text{ for all } i_1, i_2 \in V_k, \text{ for all } 1 \le k, l \le m,
\end{equation}
that is, if the connectivity from a node in $V_i$ to all nodes in $V_j$ is independent of the chosen node in $V_i$. For completeness, we show the following. 

\begin{proposition}\label{prop:equitable}
Let $V=V_1 \cup \ldots V_m$ be a partition of the vertex set of a graph with adjacency matrix $A=(a_{ij})$, and let $S$ be the characteristic matrix of the partition. Write $Q(A)$ for the quotient with respect to the partition.
\begin{itemize}
\item[(i)] The partition is equitable if and only if $AS = SQ(A)$; 
\item[(ii)] The partition into orbits of the automorphism group is equitable. 
\end{itemize}
\end{proposition}
\begin{proof} (i) 
Fix $1\le i \le n$ and $1\le k \le m$, and suppose $i \in V_l$. Then 
\begin{align*}
	[AS]_{ik} &= \sum_{j \in V_k } a_{ij},
\end{align*}
and, using the equitable condition, 
\begin{align*}
	[SQ_\text{l}(A)]_{ik} &= [Q_\text{l}(A)]_{lk} 
	= \frac{1}{|V_l|} \sum_{\substack{i_1 \in V_l\\ j\in V_k}} a_{i_1j} 
	= \frac{1}{|V_l|} |V_l| \sum_{j \in V_k } a_{ij} = \sum_{j \in V_k } a_{ij}.
\end{align*}
For the converse, note that $[AS]_{il}$ does not depend on $i$ but on the orbit of $i$. Namely, given $i_1, i_2 \in V_k$,
\begin{align*}
	\sum_{j \in V_l} a_{i_1j} = [AS]_{i_1l} = [Q_\text{l}(A)]_{kl} = [AS]_{i_2l} = \sum_{j \in V_l} a_{i_2j}.
\end{align*}
(ii) Given $i_1$ and $i_2$ in the same orbit $\Delta_k$, choose an automorphism $\sigma$ such that $\sigma(i_1)=i_2$. Then, since automorphisms respect the adjacency matrix, $a_{ij} = a_{\sigma(i)\sigma(j)}$ for all $i$, $j$, we have
\begin{align*}
	\sum_{j \in \Delta_l} a_{i_1j} =  \sum_{j \in \Delta_l} a_{\sigma(i_1)\sigma(j)} = \sum_{j \in \Delta_l} a_{i_2\sigma(j)} = \sum_{j \in \Delta_l} a_{i_2j},
\end{align*}
where the last equality follows from the fact that an element in a group permutes orbits, in this case, $\{j \,:\, j \in \Delta_l\} = \{\sigma(j) \,:\, j \in \Delta_l\}$. Hence the partition into orbits is equitable. 
\end{proof}

It follows immediately that the quotient eigenvalues are a subset of the eigenvalues of the parent network,
\begin{align}\label{eqn20}
	Q(A)v&=\lambda v \ \implies \ A(Sv)=SQ(A)v=\lambda Sv.
\end{align}
(Note that $Sv \neq 0$ if $v \neq 0$.) That is, the spectrum of the quotient is a subset of the spectrum of the graph, with eigenvectors lifted from the quotient by repeating entries on orbits. Moreover, we can complete an eigenbasis with eigenvectors orthogonal to the partition (adding up to zero on each orbit).
\begin{theorem}\label{thm:spectraldecomposition}
Suppose that $A$ is an $n \times n$ real symmetric matrix and $B$ the $m \times m$ quotient matrix with respect to an equitable partition $V_1\cup\ldots\cup V_m$ of the set $\{1,2,\ldots,n\}$. Let $S$ be the characteristic matrix of the partition. Then $A$ has an eigenbasis of the form 
\begin{align*}
	\left\{Sv_1,\ldots,Sv_{m},w_1,\ldots,w_{n-m}\right\}, 
\end{align*}
where $\{v_1,\ldots,v_{m}\}$ is any eigenbasis of $B$, and $S^Tw_i=0$ for all $i$.
\end{theorem}
\begin{proof}
Recall that $Sv \neq 0$ if $v\neq 0$ ($S$ lifts the vector $v$ from the quotient by repeating entries on each orbit) so the linear map 
\begin{align*}
	\mathbb{R}^m \to \mathbb{R}^n, v \mapsto Sv
\end{align*}
has trivial kernel and hence it is an isomorphism onto its image. In particular, $\mathcal B = \{Sv_1,\ldots,Sv_m\}$ is also a linearly independent set, and they are all eigenvectors of $A$, since $AS=SB$ as the partition is equitable. 
To finish the proof we need to complete $\mathcal B$ to a basis $\left\{Sv_1,\ldots,Sv_{m},w_1,\ldots,w_{n-m}\right\}$ such that each $w_j$ is an $A$-eigenvector orthogonal to all $Sv_i$. As $\mathcal B$ is a basis of $\text{Im}(S)$, this would imply $w_i \in \text{Im}(S)^\perp = \text{Ker}(S^T)$, giving $S^Tw_i=0$ for all $i$, as desired. 
Since $A$ is diagonalisable, $\mathbb{R}^n$ decomposes as an orthogonal direct sum of eigenspaces, $\mathbb{R}^n=\bigoplus_\lambda E_\lambda$. In each $E_\lambda$, we can find vectors $w_j$ such that they complete $V_\lambda=\{ Sv_i \in \mathcal B \,|\, v_i \text{ $\lambda$-eigenvector}\}$ to a basis of $E_\lambda$ and that are orthogonal to all vectors in $V_\lambda$ (consider the orthogonal complement of the subspace generated by $V_\lambda$ in $E_\lambda$). Repeating this procedure on each $E_\lambda$, we find vectors $\{w_1,\ldots,w_{n-m}\}$ as needed. 
\end{proof}

The statement and proof above holds for arbitrary matrices $A$ by replacing `eigenbasis' by `maximal linearly independent set' and removing the condition $S^Tw_i=0$. It would be interesting to know whether the condition $S^Tw_i=0$ holds for motif eigenvectors in the directed case as well (the proof above is no longer valid). 



Further to the spectral decomposition theorem above, we can give an even more precise description of the redundant spectrum: it is made of the contributions from the spectrum of each individual symmetric motif.


\begin{theorem}\label{thm:redeig}
Let $\mathcal M$ be a symmetric motif of a (possibly weighted) undirected graph $\mathcal G$. If $(\lambda,w)$ is a redundant eigenpair of $\mathcal M$ then $(\lambda, \widetilde{w})$ is a eigenpair of $\mathcal G$, where $\widetilde{w}$ is equal to $w$ on (the vertices of) $\mathcal M$, and zero elsewhere.
\end{theorem}
\begin{proof}
Since $(\lambda,v)$ is an $\mathcal M$-eigenpair,
\begin{equation*}
	\sum_{j \in V(\mathcal M)} [A_\mathcal{M}]_{ij} w_j = \lambda w_i \quad \forall\, i \in V(\mathcal M),
\end{equation*}
where $A_\mathcal{M}$ is the adjacency matrix of $\mathcal M$. We can decompose $\mathcal M$ into orbits, 
\begin{equation*}
	V(\mathcal M) = V_1 \cup \ldots \cup V_m,
\end{equation*}
and, by the spectral decomposition theorem above applied to $\mathcal M$, $w$ is orthogonal to each orbit, that is,
\begin{equation*}
	\sum_{j \in V_i} w_j = 0 \quad \forall\, 1 \le i \le m .
\end{equation*}
We need to show that $(\lambda,\widetilde{w})$ is a $\mathcal G$-eigenpair. Let us write $A$ for the adjacency matrix of $\mathcal G$ (recall $\mathcal M$ is a subgraph so $A$ restricts to $A_\mathcal{M}$ on $\mathcal{M}$). We need to show $A \widetilde{w}=\lambda\widetilde{w}$. Given $i \in V(\mathcal G)$, we have two cases. First, if $i \in V(\mathcal M)$,
\begin{align*}
	\sum_{j \in V(\mathcal G)} [A]_{ij} \widetilde{w}_j & = 
	\sum_{j \in V(\mathcal M)} [A]_{ij} \widetilde{w}_j + \sum_{j \in V(\mathcal G) \setminus V(\mathcal M)} [A]_{ij} \widetilde{w}_j \nonumber\\ 
	&= \sum_{j \in V(\mathcal M)} [A]_{ij} w_j = \lambda w_i = \lambda \widetilde{w}_i,
\end{align*}
since $\widetilde{w}$ equals $w$ on $\mathcal M$, and is zero outside $\mathcal M$. The second case, when $i \in V(\mathcal G) \setminus V(\mathcal M)$, gives
\begin{align*}
	\sum_{j \in V(\mathcal G)} [A]_{ij} \widetilde{w}_j = 
	\sum_{j \in V(\mathcal M)} [A]_{ij} w_j,
\end{align*}
as before, and then we use the decomposition of $\mathcal M$ into orbits,
\begin{align*}
	\sum_{j \in V(\mathcal M)} [A]_{ij} {w}_j = 
	\sum_{k=1}^m \sum_{j \in V_k} [A]_{ij} {w}_j = 
	\sum_{k=1}^m \alpha_k \sum_{j \in V_k} {w}_j \,.
\end{align*}
Here we have used that the vertex $i$, outside the motif, connects uniformly to each orbit, that is, $A_{ij_1}=A_{ij_2}$ for all $j_1,j_2 \in V_k$, and we call this quantity $\alpha_k$. Finally, recall that $w$ is orthogonal to each orbit, to conclude
\begin{gather*}
	\sum_{j \in V(\mathcal M)} [A]_{ij} {w}_j =
	\sum_{k=1}^m \alpha_k \sum_{j \in V_k} {w}_j =
	0 = \lambda \widetilde{w}_i\,. \qedhere
\end{gather*}
\end{proof}
Therefore, the redundant spectrum of $\mathcal G$ is the union of the redundant eigenvalues of the symmetric motifs, together with their redundant eigenvectors localised on them. Since most symmetric motifs in real-world networks are basic, most symmetric motifs in the network representation of a network measure will be basic too. Given their constrained structure, one can in fact determine the redundant spectrum of BSMs with up to few orbits, for arbitrary undirected networks with symmetry. 

\smallskip
\noindent\textbf{Redundant spectrum of a 1-orbit BSM.}  
A BSM with one orbit is an $(\alpha,\beta)$-uniform graph $K_n^{\alpha,\beta}$ with adjacency matrix $A_n^{\alpha,\beta}=(a_{ij})$ given by $a_{ij}=\alpha$ and $a_{ii}=\beta$ for all $i\neq j$. Then $K_n^{\alpha,\beta}$ has eigenvalues $(n-1)\alpha+\beta$ (non-redundant), with multiplicity 1, and $-\alpha + \beta$ (redundant), with multiplicity $n-1$. The corresponding eigenvectors are $\mathbf{1}$, the constant vector 1 (non-redundant), and $\mathbf{e}_i$, the vectors with non-zero entries 1 at position 1, and $-1$ at position $i$, $2\le i \le n$ (redundant). This can be shown directly by computing $A_n^{\alpha,\beta}\mathbf{1}$ and $A_n^{\alpha,\beta}\mathbf{e}_i$, and noting that $\mathbf{1}$, $\mathbf{e}_2$, \ldots,  $\mathbf{e}_n$ are linearly independent (although not orthogonal) and thus form an eigenbasis. Indeed, $A_n^{\alpha,\beta}\mathbf{1}$ is the vector of column sums of the matrix $A_n^{\alpha,\beta}$, which are constant $(n-1)\alpha+\beta$, and $A_n^{\alpha,\beta}\mathbf{e}_i$ is the constant 0 vector, except possibly at positions 1, which equals $\beta -\alpha$, and $i$, which equals $\alpha-\beta$. 


\smallskip
\noindent\textbf{Redundant spectrum of a 2-orbit BSM.} 
A BSM with two orbits is a uniform join of the form 
\begin{equation}\label{eq:BSM2}
	K_n^{\alpha_1,\beta_1} \stackrel{\gamma,\delta}{\longleftrightarrow} K_n^{\alpha_2,\beta_2}\,.
\end{equation}
Define $a=\alpha_1-\beta_1$, $b=\alpha_2-\beta_2$, $c=\gamma-\delta$, and note that $c \neq 0$: otherwise $\gamma=\delta$ and we can freely permute one orbit while fixing the other, that is, this would not be a BSM with two orbits but rather two BSMs with one orbit each. As above, let $\mathbf{e}_i$ be the vector with non-zero entries 1 at position 1, and $-1$ at position $i$, for any $2\le i \le n$. 
\begin{lemma}
The following set of vectors is linearly independent
\begin{align*}
	\{(\kappa_1\, e_i \,|\, e_i),  (\kappa_2\, e_i \,|\, e_i) \mid 2 \le i \le n\} 
\end{align*}
for all $\kappa_1 \neq \kappa_2 \in \mathbb{R}$.
\end{lemma}
\begin{proof}
Define the $(n-1)\times n$ matrix 
\[
	B_n = \begin{pmatrix} \mathbf{1} \mid - \textup{Id}_{n-1} \end{pmatrix}
\]
where $\mathbf{1}$ is a constant 1 column vector, and $\textup{Id}_{n-1}$ the identity matrix of size $n-1$. The set of vectors in the statement can be arranged in block matrix form as 
\begin{align*}
	\left( \begin{array}{c|c} \kappa_1 \, B_n & B_n\\ \hline
				\kappa_2 \, B_n & B_n \end{array}\right).
\end{align*}
This matrix has a minor of order $2(n-1)$, 
\begin{align*}
	\det \left( \begin{array}{c|c} -\kappa_1 \, \textup{Id}_{n-1} & -\textup{Id}_{n-1}\\ \hline	-\kappa_2 \, \textup{Id}_{n-1} & -\textup{Id}_{n-1} \end{array}\right).
\end{align*}
Using that $\det\left(\begin{smallmatrix} A & B \\ C & D\end{smallmatrix}\right) = AD-BC$ whenever $A$, $B$, $C$, $D$ are square blocks of the same size and $C$ commutes with $D$ \cite{horn1990matrix}, this minor equals
\begin{gather*}
	\det \begin{pmatrix}  -\kappa_1 \, \textup{Id}_{n-1} -  \kappa_2 \, \textup{Id}_{n-1}\end{pmatrix} = (-1)^{n-1}(\kappa_1 + \kappa_2)^{n-1} \neq 0 \iff \kappa_1 \neq \kappa_2. \qedhere
\end{gather*}
\end{proof}
Next we derive conditions for a vector $v_i=(\kappa e_i | e_i)$ to be an eigenvector of the uniform join \eqref{eq:BSM2}, that is, $A v_i = \lambda v_i$, for some $\lambda \in \mathbb{R}$, where $A$ is the (symmetric) adjacency matrix of the uniform join, 
\begin{align}
	A = \left(\begin{array}{c|c} A_n^{\alpha_1, \beta_1} & A_n^{\gamma, \delta} \\ \hline
  	 		 			A_n^{\gamma, \delta} & A_n^{\alpha_2, \beta_2} \end{array} \right).
\end{align}
The $j$th entry of the vector $A v$ is
\begin{align*}
	\kappa \beta_1 - \kappa \alpha_1 + \delta - \gamma &= -(\kappa a + c)& \ & j=1\\
	\kappa \alpha_1 - \kappa \beta_1 + \gamma - \delta &= \kappa a + c& \ & j=i\\
	\kappa \delta - \kappa \gamma + \beta_2 - \alpha_2 &= -(\kappa c + b) & \ & j=n+1\\
	\kappa \gamma - \kappa \delta + \alpha_2 - \beta_2 &= \kappa c + b& \ & j=n+i\\
	& 0 & \ & \text{otherwise.}
\end{align*}
Comparing these with the entries of the vector $\lambda v_i$, we obtain 
\begin{align}
	Av_i=\lambda v_i \iff \left(\kappa a + c = - \lambda \kappa \ \text{ and } \ \kappa c + b = - \lambda\right),
\end{align}
The two equations on the right-hand side are satisfied if and only if $\lambda=-\kappa c - b$ and $\kappa$ is a solution of the quadratic equation 
\begin{align}
	c \kappa^2 + (b-a)\kappa - c = 0,
\end{align}
which has two distinct real solutions
\begin{align}
	\frac{(a-b) \pm \sqrt{(a-b)^2 + 4 c^2}}{2c},
\end{align}
since $c\neq 0$, as explained above. Together with the lemma, we have shown the following.

\begin{theorem}\label{thm:RSpec}
The redundant spectrum of a symmetric motif with two orbits $K_n^{\alpha_1,\beta_1} \stackrel{\gamma,\delta}{\longleftrightarrow} K_n^{\alpha_2,\beta_2}$ is given by the eigenvalues 
\begin{align*}
	\lambda_1 &= -b-c\kappa_1 = \frac{-(a+b) + \sqrt{(a-b)^2+4c^2}}{2},  \text{ and},\\
	\lambda_2 &= -b-c\kappa_2 = \frac{-(a+b) - \sqrt{(a-b)^2+4c^2}}{2},
\end{align*}
each with multiplicity $n-1$, and eigenvectors $(\kappa_1 e_i | e_i)$ and $(\kappa_2 e_i | e_i)$ respectively, where $\kappa_1$ and $\kappa_2$ are the two solutions of the quadratic equation $c \kappa^2 + (b-a)\kappa - c = 0$, $a=\alpha_1-\beta_1$, $b=\alpha_2-\beta_2$ and $c=\gamma-\delta \neq 0$. 
\end{theorem}


For unweighted graphs without loops, we recover the redundant eigenvalues for BSMs with two orbits predicted in \cite{macarthur2009spectral}, as follows. We have $\beta_1=\beta_2=0$, $\alpha_1, \alpha_2, \gamma, \delta \in \{0,1\}$ and thus $a, b \in \{0,1\}$ and $c \in \{ -1, 1\}$. If $a=b$, the quadratic equation becomes $\kappa^2-1=0$ with solutions $\kappa=\pm1$ and thus $\lambda = -b-c\kappa \in \{ -2,-1,0,1\}$. If $a \neq b$ we can assume $a=1$, $b=0$ and the quadratic $c\kappa - \kappa - c = 0$ has solutions $\varphi$ and $1-\varphi$ if $c=1$, $-\varphi$ and $\varphi-1$ if $c=-1$, where $\varphi=\frac{1+\sqrt{5}}{2}$ is the golden ratio. In either case, the redundant eigenvalues $\lambda=-b-c\kappa=-c\kappa$ are $-\varphi$ and $\varphi-1$. Altogether, the redundant eigenvalues for 2-orbit BSMs are $\{-2, -\varphi, -1, 0, \varphi-1, 1 \}$, which equals the redundant eigenvalues $\text{RSpec}_2$ in the notation of \cite{macarthur2009spectral}. 


We omit the calculation of the redundant spectrum of BSMs with three (or more) orbits, as it becomes much more elaborate, and its relevance in real-world networks is less justified (for example, less than 1\% of BSMs in each of our test networks, Table \ref{Table1}, has 3 or more orbits).

\subsection*{Applications}
\smallskip
\begin{theorem}[Communicability]\label{thm:comm}
Let $Q(A)$ be the quotient of the adjacency matrix $A$ of a network with respect to the partition into orbits of the automorphism group. Let $f(x)= \sum a_n x^n$ be an analytic function. Then $f(Q(A))=Q(f(A))$. 
\end{theorem}
\begin{proof}
Call $B=Q(A)$ and recall that $AS=SB$ by Proposition \ref{prop:equitable}(i), where $S$ is the characteristic matrix of the partition. Therefore, $A^nS=SB^n$ for all $n\ge 0$ and
\begin{align}
	Q\left(f(A)\right) &= \Lambda^{-1}S^T \left( \sum_{n=0}^\infty a_n A^n \right) S = \sum_{n=0}^\infty a_n \left( \Lambda^{-1} S^T A^n S \right)\\ 
	 &= \sum_{n=0}^\infty a_n \left( \Lambda^{-1} S^T S B^n \right) = \sum_{n=0}^\infty a_n B^n = f(B),
\end{align}
since $\Lambda^{-1} S^T S$ is the identity matrix. 
\end{proof}

\begin{theorem}[Shortest path distance]\label{thm:SPD}
Let $A=(a_{ij})$ be as above. Then
\begin{itemize}
\item[(i)] if $(v_1,v_2,\ldots,v_{n})$ is a shortest path from $v_1$ to $v_n$ and $\sigma \in \text{Aut}(\mathcal G)$, then $\left(\sigma(v_1),\sigma(v_2),\ldots,\sigma(v_n)\right)$ is a shortest path from $\sigma(v_1)$ to $\sigma(v_n)$;
\item[(ii)] if $(v_1,v_2,\ldots,v_n)$ is a shortest path from $v_1$ to $v_n$, and $v_1$ and $v_n$ belong to different symmetric motifs, then $v_i$ and $v_{i+1}$ belong to different orbits, for all $1 \le i \le n-1$;
\item[(iii)] if $u$ and $v$ belong to orbits $U$, respectively $V$, in different symmetric motifs, then the distance from $u$ to $v$ in $\mathcal G$ equals the distance from $U$ to $V$ in the unweighted (or skeleton) quotient $\mathcal Q$.
\end{itemize}
\end{theorem}
\begin{proof}(i) Since automorphisms are bijections and preserve adjacency, $(\sigma(v_1),\sigma(v_2),\ldots,\sigma(v_n))$ is a path from $\sigma(u)$ to $\sigma(v)$ of the same length. If there were a shorter path $(\sigma(u)=w_1,w_2,\ldots,\sigma(v)=w_m)$, $m < n$, the same argument applied to $\sigma^{-1}$ gives a shorter path $(u=\sigma^{-1}(w_1),\sigma^{-1}(w_2),\ldots,v=\sigma^{-1}(w_m))$ from $u$ to $v$, a contradiction.

(ii) Any subpath of a minimal length path is also of minimal length between its endpoints. Arguing by contradiction, there exists a subpath $p=(w_1,w_2,\ldots,w_n)$ (or $p=(w_n,w_{n-1},\ldots,w_1)$), such that $w_1$ and $w_2$ belong to the same orbit, and $w_n$ belongs to a different symmetric motif. Hence, we can find $\sigma \in \text{Aut}(\mathcal G)$ with $\sigma(w_2)=w_1$ and fixing $w_n$. This implies $\sigma(p)=(\sigma(w_1),\sigma(w_2)=w_1,\sigma(w_3),\ldots,\sigma(w_n)=w_n)$, a shortest path by (i), of length $n-1$. The subpath $(w_1,\sigma(w_3),\ldots,w_n)$ has length $n - 2$, contradicting $p$ being a minimal length path from $w_1$ to $w_n$. (The case $p=(w_n,w_{n-1},\ldots,w_1)$ is analogous.) 

(iii) Let $p=(u=v_1,v_2,\ldots,v_{n+1}=v)$ be a shortest path from $u$ to $v$, so that $d^{\mathcal G}(u,v) = n$. Let $V_k$ be the orbit containing $v_k$, for all $k$. By (ii), $V_k \neq V_{k+1}$ for all $1 \le k \le n$ thus $q=(U=V_1,V_2,\ldots,V_{n+1}=V)$ is a path in $\mathcal Q$ and $d^{\mathcal Q}(U,V) \le n$. By contradiction, assume there is a shorter path in $\mathcal Q$ from $U$ to $V$, that is, $(U=W_1,W_2,\ldots,W_{m+1}=V)$ with $m < n$. The we can construct a path in $\mathcal G$ from $u$ to $v$ of length $m$ (a contradiction), as follows. For each $1 \le i \le m$, $W_i$ is connected to $W_{i+1}$ in $\mathcal Q$, hence there is a vertex in $W_i$ connected to at least one vertex in $W_{i+1}$. Since vertices in an orbit are structurally indistinguishable, \emph{any} vertex in $W_i$ is then connected to at least one vertex in $W_{i+1}$ (formally, if $w\in W_i$ is connected to $w'\in W_{i+1}$ then $\sigma(w)\in W_i$ is connected to $\sigma(w')\in W_{i+1}$). This allows us to construct a path in $\mathcal G$ from $u$ to $v$ of length $m < n$, a contradiction.
\end{proof}










Let us call the \emph{external degree} of a vertex as the number of adjacent vertices outside the motif it belongs to. The proof of the following is straightforward from the definitions. 
\begin{theorem}[Symmetric motif Laplacian]\label{thm:laplacian}
A symmetric motif $\mathcal M$ in $\mathcal G$ induces a symmetric motif in $\mathcal L$ with adjacency matrix
\begin{align}
	L_{\mathcal M} + \left(d_1I_{m_1} \oplus \ldots \oplus d_k I_{m_k}\right),
\end{align}
where $L_{\mathcal M}$ is the ordinary Laplacian matrix of $\mathcal M$ considered as a graph on its own, and $d_1, \ldots, d_k$ are the external degrees of the $k$ orbits of $\mathcal M$ of sizes $m_1, \ldots, m_k$. (Here $I_n$ is the identity matrix of size $n$ and we use $\oplus$ to construct a block diagonal matrix.)
\end{theorem}

Recall that each orbit in a BSM (in an undirected, unweighted graph) is either a complete or an empty graph. 
\begin{corollary}[Redundant Laplacian eigenvalues]\label{cor:laplacian}
Let $\mathcal G$ be an undirected, unweighted network. If $\mathcal M$ is a 1-orbit BSM with $m$ vertices of external degree $d$, then the redundant Laplacian eigenvalue induced by $\mathcal M$ is $d$ if $\mathcal M$ is an empty graph, and $d+m$ if $\mathcal M$ is a complete graph, in both cases with multiplicity $m-1$.   
\end{corollary}
\begin{proof} 
By Theorem \ref{thm:laplacian}, the Laplacian of the motif in $\mathcal L$ is $L_{\mathcal M} + d I_m$. The redundant eigenvalues of this matrix are the redundant eigenvalues of $L_{\mathcal M}$ (0 if $\mathcal M$ is empty and $m$ if $\mathcal M$ is a complete graph, in both cases with multiplicity $m-1$) plus $d$. All in all, the redundant eigenvalues for 1-orbit BSMs occur a the positive integers $\mathbb{Z}^+$.
\end{proof}

\begin{theorem}[Vertex compression]\label{thm:vertexcompression}
If $\mathbf{v}$ is a vector of length $n_\mathcal{G}$ that is constant on orbits, then 
$
	S\Lambda^{-1}S^T \mathbf{v} = \mathbf{v}.
$
\end{theorem}
\begin{proof}
First, note that $S^TS = \Lambda$ (this holds for any partition of the vertex set). 
\begin{align*}
	[S^TS]_{\alpha \beta} = \sum_i [S^T]_{\alpha i} [S]_{i \beta} 
	= \sum_i [S]_{i \alpha} [S]_{i \beta} = 
	\begin{cases} 
	0 & \text{if $\alpha \neq \beta$,}\\ 
	|V_\alpha| & \text{if $\alpha=\beta$.}
	\end{cases}
\end{align*}
As $v$ is constant on orbits, it is already of the form $v=Sw$ for some $w$. Therefore 
	$S\Lambda^{-1}S^T v = S\Lambda^{-1}S^T Sw = Sw = v$. 
\end{proof}

\begin{proposition}[Degree centrality] \label{prop:degreequotient}
Let $B=(b_{\alpha \beta})$ be the adjacency matrix of the quotient, and $V=V_1\cup\ldots\cup V_m$ the partition of the vertex set into orbits. If $i \in V_\alpha$, then $d_i^\mathcal{G}=d_\alpha^{\mathcal{Q}, \text{out}}$.
\end{proposition}
\begin{proof}
\begin{align*}
	d_i^\mathcal{G} &= \sum_{j \in V} a_{i j} = 
	\sum_{j \in V_1} a_{i j} + \ldots + 
	\sum_{j \in V_m} a_{i j}
	= \frac{1}{n_1} \sum_{\substack{j \in V_1\\i \in V_i}} a_{i j} + \ldots + 
	\frac{1}{n_m}\sum_{\substack{j \in V_m\\i \in V_i}} a_{i j}\\
	&= b_{i1}+\ldots+b_{im} = d_\alpha^{\mathcal{Q}, \text{out}}. \qedhere
\end{align*}
\end{proof}


\subsection*{Weighted and directed networks}
The adjacency matrix of a network can encode arbitrary weights and directions, making a general $n \times n$ real matrix $A$ the adjacency matrix of some (weighted, directed) network. The definition of automorphism group, geometric decomposition, symmetric motif, and orbit, and their properties, as they are defined only in terms of $A$, carry verbatim to arbitrarily weighted and directed networks. In this setting, a symmetry (automorphism), respects not only adjacency, but weights and directions. In particular, the automorphism group is smaller than (a subgroup of) the automorphism group of the underlying undirected, unweighted network. By introducing edge weights or directions, some symmetries will disappear, removing (and occasionally subdividing) symmetric motifs and orbits, as the next results shows.  
\begin{theorem}\label{prop:weighteddirected}
Let $A_\textup{w}=(w_{ij})$ be the adjacency matrix of an arbitrarily weighted and directed network $\mathcal G_\textup{w}$, and $A = (a_{ij})$ the adjacency matrix of the underlying undirected and unweighted network $\mathcal G$, that is, $a_{ij}=\textup{sgn}(|w_{ij}|+|w_{ji}|)$. Consider the symmetric motifs of $\mathcal G$, respectively $\mathcal G_\textup{w}$, with vertex sets $M_1$, \ldots, $M_m$, respectively $M'_1$, \ldots, $M'_{m'}$. Then for every $1 \le i \le m'$ there is a unique $1 \le j \le m$ such that $M_i' \subseteq M_j$. Similarly, each vertex orbit in $\mathcal G_\textup{w}$ is a subset of a vertex orbit in $\mathcal G$. 
\end{theorem}

\begin{proof}
First we show that the automorphism group of $\mathcal G_\textup{w}$ is a subgroup of the automorphism group of $\mathcal G$. If $\sigma \colon V \to V$ is a permutation of the vertices, then 
\[
	w_{\sigma(i)\sigma(j)} = w_{ij} \implies a_{\sigma(i)\sigma(j)} = a_{ij}
\]
by considering two cases: $w_{ij} \neq 0$ implies $w_{\sigma(i)\sigma(j)} \neq 0$ which gives $a_{ij} = a_{\sigma(i)\sigma(j)} = 1$; $w_{ij} = 0$ implies $w_{\sigma(i)\sigma(j)} = 0$ which gives $a_{ij} = a_{\sigma(i)\sigma(j)} = 0$ (note $w_{ij} \neq 0 \iff a_{ij} =1$). Hence $\textup{Aut}(\mathcal G_\textup{w}) \subset \textup{Aut}(\mathcal G)$, which immediately gives the result on orbits. 

Let us choose essential\cite{macarthur2008symmetry} sets of generators $S$, respectively $S'$, of $\textup{Aut}(\mathcal G)$, respectively $\textup{Aut}(\mathcal G_\textup{w})$, with support-disjoint partitions
\begin{align*}
	X = X_1 \cup \ldots \cup X_m, \text{ respectively} \
	X' = X'_1 \cup \ldots \cup X'_{m'}.
\end{align*}
It is enough to prove the statement for these sets: given $i$, there is unique $j$ such that $X'_i \subseteq X_j$. Let $x' \in X'_i \subseteq \textup{Aut}(\mathcal G_\textup{w}) \subseteq \textup{Aut}(\mathcal G)$ thus we can write $x' = h_1\cdot\ldots\cdot h_m$ with $h_k \in H_k = \langle X_k \rangle$. Since $X'$ is an essential set of generators, there is an index $j$ such that $h_k=1$ (the identity, or trivial permutation) for all $k \neq j$, so that $x'=h_j$. Given any other $y' \in X'_i$, the same argument gives $y'=h_l$ for some $1\le l \le m$. We claim $j=l$, as follows. The partition of $X$, respectively $X'$, above are the equivalence classes of the equivalence relation generated by $\sigma \sim \tau$ if $\sigma$ and $\tau$ are not support-disjoint permutations. Since $x', y'$ are in the same equivalence class, so are $h_j$ and $h_l$ and thus $j=l$. 
\end{proof}
The same result applies to networks with other additional structure, not necessarily expressed in terms of the adjacency matrix, such as arbitrary vertex or edge labels, by restricting to automorphisms preserving the additional structure. We obtain fewer symmetries, and a refinement of the geometric decomposition, symmetric motifs, and orbits as above. The results in this paper, although applicable in theory, become less useful in practice as further restrictions are imposed, reducing the number of available network symmetries.


\subsection*{Asymmetric measures}
In the case of an asymmetric network measure ($F(i,j)\neq F(j,i)$), its network representation $F(\mathcal G)$ is directed even if $\mathcal G$ is not. However, $F(\mathcal G)$ still inherits all the symmetries of $\mathcal G$, that is, every automorphism of $\mathcal G$ respects weights \emph{and} edge directions in $F(\mathcal G)$. Therefore, $F(\mathcal G)$ has the same symmetric motifs (as vertex sets) and orbits as $\mathcal G$, and the structural results in this paper apply verbatim.

\section*{Data and code availability}
The datasets analysed during the current study are available at the locations stated in the caption to Table \ref{Table1}. The datasets generated during the current study can be found at \url{https://doi.org/10.6084/m9.figshare.11619792} and the code used to process the datasets at \url{https://bitbucket.org/rubenjsanchezgarcia/networksymmetry/}.

{\scriptsize
\bibliography{NetworkFunctions}}

\begin{thebibliography}{10}
\expandafter\ifx\csname url\endcsname\relax
  \def\url#1{\texttt{#1}}\fi
\expandafter\ifx\csname urlprefix\endcsname\relax\def\urlprefix{URL }\fi
\expandafter\ifx\csname doiprefix\endcsname\relax\def\doiprefix{DOI }\fi
\providecommand{\bibinfo}[2]{#2}
\providecommand{\eprint}[2][]{\url{#2}}

\bibitem{newman2010networks}
\bibinfo{author}{Newman, M.}
\newblock \emph{\bibinfo{title}{{Networks: An Introduction}}}
  (\bibinfo{publisher}{Oxford University Press}, \bibinfo{address}{Oxford},
  \bibinfo{year}{2010}).

\bibitem{watts1998collective}
\bibinfo{author}{Watts, D.} \& \bibinfo{author}{Strogatz, S.}
\newblock \bibinfo{journal}{\bibinfo{title}{Collective dynamics of small-world
  networks}}.
\newblock {\emph{\JournalTitle{Nature}}} \textbf{\bibinfo{volume}{393}},
  \bibinfo{pages}{440--442} (\bibinfo{year}{1998}).

\bibitem{barabasi1999emergence}
\bibinfo{author}{Barab{\'a}si, A.-L.} \& \bibinfo{author}{Albert, R.}
\newblock \bibinfo{journal}{\bibinfo{title}{Emergence of scaling in random
  networks}}.
\newblock {\emph{\JournalTitle{Science}}} \textbf{\bibinfo{volume}{286}},
  \bibinfo{pages}{509--512} (\bibinfo{year}{1999}).

\bibitem{milo2002network}
\bibinfo{author}{Milo, R.} \emph{et~al.}
\newblock \bibinfo{journal}{\bibinfo{title}{Network motifs: simple building
  blocks of complex networks}}.
\newblock {\emph{\JournalTitle{Science}}} \textbf{\bibinfo{volume}{298}},
  \bibinfo{pages}{824--827} (\bibinfo{year}{2002}).

\bibitem{tononi1999measures}
\bibinfo{author}{Tononi, G.}, \bibinfo{author}{Sporns, O.} \&
  \bibinfo{author}{Edelman, G.~M.}
\newblock \bibinfo{journal}{\bibinfo{title}{Measures of degeneracy and
  redundancy in biological networks}}.
\newblock {\emph{\JournalTitle{{Proc Natl Acad Sci USA}}}}
  \textbf{\bibinfo{volume}{96}}, \bibinfo{pages}{3257--3262}
  (\bibinfo{year}{1999}).

\bibitem{albert2002statistical}
\bibinfo{author}{Albert, R.} \& \bibinfo{author}{Barab{\'a}si, A.-L.}
\newblock \bibinfo{journal}{\bibinfo{title}{Statistical mechanics of complex
  networks}}.
\newblock {\emph{\JournalTitle{{Rev Mod Phys}}}} \textbf{\bibinfo{volume}{74}},
  \bibinfo{pages}{47} (\bibinfo{year}{2002}).

\bibitem{chung2003duplication}
\bibinfo{author}{Chung, F.}, \bibinfo{author}{Lu, L.}, \bibinfo{author}{Dewey,
  T.~G.} \& \bibinfo{author}{Galas, D.~J.}
\newblock \bibinfo{journal}{\bibinfo{title}{Duplication models for biological
  networks}}.
\newblock {\emph{\JournalTitle{{J Comput Biol}}}}
  \textbf{\bibinfo{volume}{10}}, \bibinfo{pages}{677--687}
  (\bibinfo{year}{2003}).

\bibitem{xiao2008emergence}
\bibinfo{author}{Xiao, Y.}, \bibinfo{author}{Xiong, M.}, \bibinfo{author}{Wang,
  W.} \& \bibinfo{author}{Wang, H.}
\newblock \bibinfo{journal}{\bibinfo{title}{Emergence of symmetry in complex
  networks}}.
\newblock {\emph{\JournalTitle{{Phys Rev E}}}} \textbf{\bibinfo{volume}{77}},
  \bibinfo{pages}{066108} (\bibinfo{year}{2008}).

\bibitem{nishikawa2016network}
\bibinfo{author}{Nishikawa, T.} \& \bibinfo{author}{Motter, A.~E.}
\newblock \bibinfo{journal}{\bibinfo{title}{Network-complement transitions,
  symmetries, and cluster synchronization}}.
\newblock {\emph{\JournalTitle{Chaos}}} \textbf{\bibinfo{volume}{26}},
  \bibinfo{pages}{094818} (\bibinfo{year}{2016}).

\bibitem{Klickstein2018generating}
\bibinfo{author}{Klickstein, I.} \& \bibinfo{author}{Sorrentino, F.}
\newblock \bibinfo{journal}{\bibinfo{title}{Generating symmetric graphs}}.
\newblock {\emph{\JournalTitle{Chaos: An Interdisciplinary Journal of Nonlinear
  Science}}} \textbf{\bibinfo{volume}{28}}, \bibinfo{pages}{121102}
  (\bibinfo{year}{2018}).
\newblock \doiprefix 10.1063/1.5064375.

\bibitem{macarthur2008symmetry}
\bibinfo{author}{MacArthur, B.~D.}, \bibinfo{author}{S{\'a}nchez-Garc{\'\i}a,
  R.~J.} \& \bibinfo{author}{Anderson, J.~W.}
\newblock \bibinfo{journal}{\bibinfo{title}{Symmetry in complex networks}}.
\newblock {\emph{\JournalTitle{{Discrete Appl Math}}}}
  \textbf{\bibinfo{volume}{156}}, \bibinfo{pages}{3525--3531}
  (\bibinfo{year}{2008}).

\bibitem{xiao2008network}
\bibinfo{author}{Xiao, Y.}, \bibinfo{author}{MacArthur, B.~D.},
  \bibinfo{author}{Wang, H.}, \bibinfo{author}{Xiong, M.} \&
  \bibinfo{author}{Wang, W.}
\newblock \bibinfo{journal}{\bibinfo{title}{Network quotients: Structural
  skeletons of complex systems}}.
\newblock {\emph{\JournalTitle{{Phys Rev E}}}} \textbf{\bibinfo{volume}{78}},
  \bibinfo{pages}{046102} (\bibinfo{year}{2008}).

\bibitem{macarthur2009spectral}
\bibinfo{author}{MacArthur, B.~D.} \& \bibinfo{author}{S{\'a}nchez-Garc{\'\i}a,
  R.~J.}
\newblock \bibinfo{journal}{\bibinfo{title}{Spectral characteristics of network
  redundancy}}.
\newblock {\emph{\JournalTitle{{Phys Rev E}}}} \textbf{\bibinfo{volume}{80}},
  \bibinfo{pages}{026117} (\bibinfo{year}{2009}).

\bibitem{pecora2014cluster}
\bibinfo{author}{Pecora, L.~M.}, \bibinfo{author}{Sorrentino, F.},
  \bibinfo{author}{Hagerstrom, A.~M.}, \bibinfo{author}{Murphy, T.~E.} \&
  \bibinfo{author}{Roy, R.}
\newblock \bibinfo{journal}{\bibinfo{title}{Cluster synchronization and
  isolated desynchronization in complex networks with symmetries}}.
\newblock {\emph{\JournalTitle{{Nat Commun}}}} \textbf{\bibinfo{volume}{5}},
  \bibinfo{pages}{4079} (\bibinfo{year}{2014}).

\bibitem{golubitsky2003symmetry}
\bibinfo{author}{Golubitsky, M.} \& \bibinfo{author}{Stewart, I.}
\newblock \emph{\bibinfo{title}{The symmetry perspective: from equilibrium to
  chaos in phase space and physical space}}, vol. \bibinfo{volume}{200}
  (\bibinfo{publisher}{Springer Science \& Business Media},
  \bibinfo{address}{Basel}, \bibinfo{year}{2003}).

\bibitem{stewart2003symmetry}
\bibinfo{author}{Stewart, I.}, \bibinfo{author}{Golubitsky, M.} \&
  \bibinfo{author}{Pivato, M.}
\newblock \bibinfo{journal}{\bibinfo{title}{Symmetry groupoids and patterns of
  synchrony in coupled cell networks}}.
\newblock {\emph{\JournalTitle{SIAM Journal on Applied Dynamical Systems}}}
  \textbf{\bibinfo{volume}{2}}, \bibinfo{pages}{609--646}
  (\bibinfo{year}{2003}).

\bibitem{golubitsky2005patterns}
\bibinfo{author}{Golubitsky, M.}, \bibinfo{author}{Stewart, I.} \&
  \bibinfo{author}{T{\"o}r{\"o}k, A.}
\newblock \bibinfo{journal}{\bibinfo{title}{Patterns of synchrony in coupled
  cell networks with multiple arrows}}.
\newblock {\emph{\JournalTitle{SIAM Journal on Applied Dynamical Systems}}}
  \textbf{\bibinfo{volume}{4}}, \bibinfo{pages}{78--100}
  (\bibinfo{year}{2005}).

\bibitem{aguiar2018synchronization}
\bibinfo{author}{Aguiar, M.~A.} \& \bibinfo{author}{Dias, A. P.~S.}
\newblock \bibinfo{journal}{\bibinfo{title}{Synchronization and equitable
  partitions in weighted networks}}.
\newblock {\emph{\JournalTitle{Chaos: An Interdisciplinary Journal of Nonlinear
  Science}}} \textbf{\bibinfo{volume}{28}}, \bibinfo{pages}{073105}
  (\bibinfo{year}{2018}).

\bibitem{sorrentino2017symmetries}
\bibinfo{author}{Sorrentino, F.}, \bibinfo{author}{Siddique, A.~B.} \&
  \bibinfo{author}{Pecora, L.~M.}
\newblock \bibinfo{journal}{\bibinfo{title}{Symmetries in the time-averaged
  dynamics of networks: Reducing unnecessary complexity through minimal network
  models}}.
\newblock {\emph{\JournalTitle{Chaos: An Interdisciplinary Journal of Nonlinear
  Science}}} \textbf{\bibinfo{volume}{29}}, \bibinfo{pages}{011101}
  (\bibinfo{year}{2019}).

\bibitem{nicosia2013remote}
\bibinfo{author}{Nicosia, V.}, \bibinfo{author}{Valencia, M.},
  \bibinfo{author}{Chavez, M.}, \bibinfo{author}{D{\'\i}az-Guilera, A.} \&
  \bibinfo{author}{Latora, V.}
\newblock \bibinfo{journal}{\bibinfo{title}{Remote synchronization reveals
  network symmetries and functional modules}}.
\newblock {\emph{\JournalTitle{{Phys Rev Lett}}}}
  \textbf{\bibinfo{volume}{110}}, \bibinfo{pages}{174102}
  (\bibinfo{year}{2013}).

\bibitem{sorrentino2016complete}
\bibinfo{author}{Sorrentino, F.}, \bibinfo{author}{Pecora, L.~M.},
  \bibinfo{author}{Hagerstrom, A.~M.}, \bibinfo{author}{Murphy, T.~E.} \&
  \bibinfo{author}{Roy, R.}
\newblock \bibinfo{journal}{\bibinfo{title}{Complete characterization of the
  stability of cluster synchronization in complex dynamical networks}}.
\newblock {\emph{\JournalTitle{{Sci Adv}}}} \textbf{\bibinfo{volume}{2}},
  \bibinfo{pages}{e1501737} (\bibinfo{year}{2016}).

\bibitem{sorrentino2016approximate}
\bibinfo{author}{Sorrentino, F.} \& \bibinfo{author}{Pecora, L.}
\newblock \bibinfo{journal}{\bibinfo{title}{Approximate cluster synchronization
  in networks with symmetries and parameter mismatches}}.
\newblock {\emph{\JournalTitle{Chaos: An Interdisciplinary Journal of Nonlinear
  Science}}} \textbf{\bibinfo{volume}{26}}, \bibinfo{pages}{094823}
  (\bibinfo{year}{2016}).

\bibitem{schaub2016graph}
\bibinfo{author}{Schaub, M.~T.} \emph{et~al.}
\newblock \bibinfo{journal}{\bibinfo{title}{Graph partitions and cluster
  synchronization in networks of oscillators}}.
\newblock {\emph{\JournalTitle{Chaos: An Interdisciplinary Journal of Nonlinear
  Science}}} \textbf{\bibinfo{volume}{26}}, \bibinfo{pages}{094821}
  (\bibinfo{year}{2016}).

\bibitem{pecora2017discovering}
\bibinfo{author}{Pecora, L.~M.}, \bibinfo{author}{Sorrentino, F.},
  \bibinfo{author}{Hagerstrom, A.~M.}, \bibinfo{author}{Murphy, T.~E.} \&
  \bibinfo{author}{Roy, R.}
\newblock \bibinfo{title}{Discovering, constructing, and analyzing synchronous
  clusters of oscillators in a complex network using symmetries}.
\newblock In \emph{\bibinfo{booktitle}{Advances in Dynamics, Patterns,
  Cognition}}, \bibinfo{pages}{145--160} (\bibinfo{publisher}{Springer},
  \bibinfo{year}{2017}).

\bibitem{siddique2018symmetry}
\bibinfo{author}{Siddique, A.~B.}, \bibinfo{author}{Pecora, L.},
  \bibinfo{author}{Hart, J.~D.} \& \bibinfo{author}{Sorrentino, F.}
\newblock \bibinfo{journal}{\bibinfo{title}{Symmetry- and input-cluster
  synchronization in networks}}.
\newblock {\emph{\JournalTitle{Physical Review E}}}
  \textbf{\bibinfo{volume}{97}}, \bibinfo{pages}{042217}
  (\bibinfo{year}{2018}).

\bibitem{BitBucket}
\bibinfo{organization}{Bitbucket repository}.
\newblock
  \emph{\bibinfo{title}{{\url{https://bitbucket.org/rubenjsanchezgarcia/networksymmetry/}}}}
  (\bibinfo{year}{2019}).

\bibitem{xiao2009efficiently}
\bibinfo{author}{Xiao, Y.}, \bibinfo{author}{Wu, W.}, \bibinfo{author}{Pei,
  J.}, \bibinfo{author}{Wang, W.} \& \bibinfo{author}{He, Z.}
\newblock \bibinfo{title}{Efficiently indexing shortest paths by exploiting
  symmetry in graphs}.
\newblock In \emph{\bibinfo{booktitle}{Proceedings of the 12th International
  Conference on Extending Database Technology: Advances in Database
  Technology}}, \bibinfo{pages}{493--504} (\bibinfo{organization}{ACM},
  \bibinfo{year}{2009}).

\bibitem{wang2012symmetry}
\bibinfo{author}{Wang, J.}, \bibinfo{author}{Huang, Y.}, \bibinfo{author}{Wu,
  F.-X.} \& \bibinfo{author}{Pan, Y.}
\newblock \bibinfo{journal}{\bibinfo{title}{Symmetry compression method for
  discovering network motifs}}.
\newblock {\emph{\JournalTitle{{IEEE/ACM Trans Comput Biol Bioinf}}}}
  \textbf{\bibinfo{volume}{9}}, \bibinfo{pages}{1776--1789}
  (\bibinfo{year}{2012}).

\bibitem{karalus2015symmetry}
\bibinfo{author}{Karalus, S.} \& \bibinfo{author}{Krug, J.}
\newblock \bibinfo{journal}{\bibinfo{title}{Symmetry-based coarse-graining of
  evolved dynamical networks}}.
\newblock {\emph{\JournalTitle{{Europhys Lett}}}}
  \textbf{\bibinfo{volume}{111}}, \bibinfo{pages}{38003}
  (\bibinfo{year}{2015}).

\bibitem{nyberg2015mesoscopic}
\bibinfo{author}{Nyberg, A.}, \bibinfo{author}{Gross, T.} \&
  \bibinfo{author}{Bassler, K.~E.}
\newblock \bibinfo{journal}{\bibinfo{title}{Mesoscopic structures and the
  laplacian spectra of random geometric graphs}}.
\newblock {\emph{\JournalTitle{{J Complex Networks}}}}
  \textbf{\bibinfo{volume}{3}}, \bibinfo{pages}{543--551}
  (\bibinfo{year}{2015}).

\bibitem{do2012engineering}
\bibinfo{author}{Do, A.-L.}, \bibinfo{author}{H{\"o}fener, J.} \&
  \bibinfo{author}{Gross, T.}
\newblock \bibinfo{journal}{\bibinfo{title}{Engineering mesoscale structures
  with distinct dynamical implications}}.
\newblock {\emph{\JournalTitle{{New J Phys}}}} \textbf{\bibinfo{volume}{14}},
  \bibinfo{pages}{115022} (\bibinfo{year}{2012}).

\bibitem{biggs1993algebraic}
\bibinfo{author}{Biggs, N.}
\newblock \emph{\bibinfo{title}{Algebraic graph theory}}
  (\bibinfo{publisher}{Cambridge University Press},
  \bibinfo{address}{Cambridge}, \bibinfo{year}{1993}).

\bibitem{godsil2013algebraic}
\bibinfo{author}{Godsil, C.} \& \bibinfo{author}{Royle, G.~F.}
\newblock \emph{\bibinfo{title}{Algebraic graph theory}}
  (\bibinfo{publisher}{Springer}, \bibinfo{address}{New York},
  \bibinfo{year}{2013}).

\bibitem{cvetkovic2010introduction}
\bibinfo{author}{Cvetkovic, D.~M.}, \bibinfo{author}{Rowlinson, P.} \&
  \bibinfo{author}{Simic, S.}
\newblock \emph{\bibinfo{title}{An introduction to the theory of graph
  spectra}} (\bibinfo{publisher}{Cambridge University Press},
  \bibinfo{address}{Cambridge}, \bibinfo{year}{2010}).

\bibitem{brouwer2011spectra}
\bibinfo{author}{Brouwer, A.~E.} \& \bibinfo{author}{Haemers, W.~H.}
\newblock \emph{\bibinfo{title}{Spectra of graphs}}
  (\bibinfo{publisher}{Springer}, \bibinfo{address}{New York},
  \bibinfo{year}{2011}).

\bibitem{estrada2005subgraph}
\bibinfo{author}{Estrada, E.} \& \bibinfo{author}{Rodriguez-Velazquez, J.~A.}
\newblock \bibinfo{journal}{\bibinfo{title}{Subgraph centrality in complex
  networks}}.
\newblock {\emph{\JournalTitle{{Phys Rev E}}}} \textbf{\bibinfo{volume}{71}},
  \bibinfo{pages}{056103} (\bibinfo{year}{2005}).

\bibitem{estrada2008communicability}
\bibinfo{author}{Estrada, E.} \& \bibinfo{author}{Hatano, N.}
\newblock \bibinfo{journal}{\bibinfo{title}{Communicability in complex
  networks}}.
\newblock {\emph{\JournalTitle{Phys Rev E}}} \textbf{\bibinfo{volume}{77}},
  \bibinfo{pages}{036111} (\bibinfo{year}{2008}).

\bibitem{erdos1963asymmetric}
\bibinfo{author}{Erd{\H{o}}s, P.} \& \bibinfo{author}{R{\'e}nyi, A.}
\newblock \bibinfo{journal}{\bibinfo{title}{Asymmetric graphs}}.
\newblock {\emph{\JournalTitle{Acta Mathematica Hungarica}}}
  \textbf{\bibinfo{volume}{14}}, \bibinfo{pages}{295--315}
  (\bibinfo{year}{1963}).

\bibitem{borgatti2015structural}
\bibinfo{author}{Borgatti, S.} \& \bibinfo{author}{Grosser, T.}
\newblock \bibinfo{journal}{\bibinfo{title}{Structural {E}quivalence: {M}eaning
  and {M}easures}}.
\newblock {\emph{\JournalTitle{International Encyclopedia of the Social \&
  Behavioral Sciences}}}  (\bibinfo{year}{2015}).

\bibitem{konect2016topology}
\bibinfo{title}{Internet topology network dataset -- {KONECT}}.
\newblock
  \bibinfo{howpublished}{\url{http://konect.uni-koblenz.de/networks/[openflights|
  opsahl\_powergrid|ca-AstroPh|topology|wordnet-words|com-amazon|actor-collaboration|as-skitter,|roadNet-CA|livejournal-links]}}
  (\bibinfo{year}{2016}).

\bibitem{HumanDisease}
\bibinfo{title}{{A Network Graph for Human Diseases}}.
\newblock
  \bibinfo{howpublished}{http://exploring-data.com/info/human-disease-network/}
  (\bibinfo{year}{2018}).

\bibitem{Yeast}
\bibinfo{title}{{Yeast Interactome Project}}.
\newblock
  \bibinfo{howpublished}{http://interactome.dfci.harvard.edu/S\_cerevisiae/}
  (\bibinfo{year}{2018}).

\bibitem{HumanPPI}
\bibinfo{title}{{Human Protein Reference Database, \emph{Release9\_062910}}}.
\newblock \bibinfo{howpublished}{http://www.hprd.org/} (\bibinfo{year}{2018}).

\bibitem{lovasz1993random}
\bibinfo{author}{Lov{\'a}sz, L.}
\newblock \bibinfo{title}{{Random Walks on Graphs: A Survey}}
  (\bibinfo{year}{1993}).

\bibitem{klein1993resistance}
\bibinfo{author}{Klein, D.~J.} \& \bibinfo{author}{Randi{\'c}, M.}
\newblock \bibinfo{journal}{\bibinfo{title}{Resistance distance}}.
\newblock {\emph{\JournalTitle{{J Math Chem}}}} \textbf{\bibinfo{volume}{12}},
  \bibinfo{pages}{81--95} (\bibinfo{year}{1993}).

\bibitem{brandes2001faster}
\bibinfo{author}{Brandes, U.}
\newblock \bibinfo{journal}{\bibinfo{title}{A faster algorithm for betweenness
  centrality}}.
\newblock {\emph{\JournalTitle{{J Math Sociol}}}}
  \textbf{\bibinfo{volume}{25}}, \bibinfo{pages}{163--177}
  (\bibinfo{year}{2001}).

\bibitem{saucy3}
\bibinfo{title}{saucy 3.0}.
\newblock \bibinfo{howpublished}{\url{http://vlsicad.eecs.umich.edu/BK/SAUCY/}}
  (\bibinfo{year}{2012}).

\bibitem{GAP4}
\bibinfo{title}{{GAP -- Groups, Algorithms, and Programming, Version 4.8.7}}.
\newblock \bibinfo{howpublished}{\url{http://www.gap-system.org}}
  (\bibinfo{year}{2017}).

\bibitem{liebeck1988graphs}
\bibinfo{author}{Liebeck, M.~W.}
\newblock \bibinfo{journal}{\bibinfo{title}{Graphs whose full automorphism
  group is a symmetric group}}.
\newblock {\emph{\JournalTitle{{J Aust Math Soc}}}}
  \textbf{\bibinfo{volume}{44}}, \bibinfo{pages}{46--63}
  (\bibinfo{year}{1988}).

\bibitem{rotman2012introduction}
\bibinfo{author}{Rotman, J.~J.}
\newblock \emph{\bibinfo{title}{An introduction to the theory of groups}}
  (\bibinfo{publisher}{Springer}, \bibinfo{address}{New York},
  \bibinfo{year}{2012}).

\bibitem{horn1990matrix}
\bibinfo{author}{Horn, R.~A.} \& \bibinfo{author}{Johnson, C.~R.}
\newblock \emph{\bibinfo{title}{Matrix analysis}}
  (\bibinfo{publisher}{{Cambridge University Press}},
  \bibinfo{address}{Cambridge}, \bibinfo{year}{1990}).

\end{thebibliography}

\section*{Acknowledgements}
Special thanks to Ben MacArthur for support and advice during the writing of this article. The Newton Institute in Cambridge supported the author during the programme `Theoretical foundations for statistical network analysis' (EPSRC grant  EP/K032208/1). Conor Wild's undergraduate project `Topics in Network Symmetries' contains (unpublished) results about the graph Laplacian in the context of symmetries generalised here. Thanks to Yamir Moreno and Emanuele Cozzo, whose `simple' question inadvertently prompted this lengthly answer.

\section*{Author contributions}
The author conceived the project, carried out the mathematical and numerical analysis, and wrote the article. 

\section*{Competing interests}
The author declares no competing interests.

\end{document}